\documentclass[11pt]{amsart}
\usepackage{
 amsmath, 
 amsxtra, 
 amsthm, 
 amssymb, 
 etex, 
 mathrsfs, 
 mathtools, 
 tikz-cd, 
 bbm,
 xr,
 comment,
 enumitem}
\usepackage[toc]{appendix} 
\usepackage[all]{xy}
\usepackage{hyperref}
\usepackage{url}
\usepackage{tipa}
\usepackage{dsfont}
\usepackage{ textcomp }
\usepackage{stmaryrd} 
\usepackage{amssymb}
\usepackage{mathabx} 
\usepackage{amsmath} 
\usepackage{amscd}
\usepackage{amsbsy}
\usepackage{comment, enumerate}
\usepackage{color}
\usepackage{mathtools,caption}
\usepackage{tikz-cd}
\usepackage{longtable}
\usepackage[utf8]{inputenc}
\usepackage[OT2,T1]{fontenc}
\usepackage{changepage}
\usepackage{commath,enumerate}

\usetikzlibrary{decorations.markings}
\tikzset{->-/.style={decoration={
  markings,
  mark=at position #1 with {\arrow{>}}},postaction={decorate}}}

\usetikzlibrary{calc} 

\DeclareMathOperator{\End}{End}

\DeclareMathOperator{\Gal}{Gal}

\DeclareMathOperator{\ord}{ord}

\newtheorem{theorem}{Theorem}[section]
\newtheorem*{theorem*}{Theorem}
\newtheorem{lemma}[theorem]{Lemma}

\newtheorem{proposition}[theorem]{Proposition}
\newtheorem{corollary}[theorem]{Corollary}

\newtheorem{defn}[theorem]{Definition}

\numberwithin{equation}{section}
\newtheorem{lthm}{Theorem} 

\newcommand{\tb}{\textup{b}}
\newcommand{\tg}{\textup{g}}
\theoremstyle{remark}
\newtheorem{remark}[theorem]{Remark}
\newtheorem{example}[theorem]{Example}

\newcommand{\Deck}{\mathrm{Deck}}
\newcommand{\R}{\mathbb{R}}
\newcommand\EatDot[1]{}

\newcommand{\n}{\natural}
\newcommand{\bE}{\mathbf{E}}
\newcommand{\Fp}{\FF_p}

\newcommand{\cG}{{\mathcal{G}}}

\newcommand{\cZ}{\mathcal{Z}}

\newcommand{\cX}{\mathcal{X}}
\newcommand{\QQ}{\mathbb{Q}}
\newcommand{\ZZ}{\mathbb{Z}}
\newcommand{\FF}{\mathbb{F}}

\newcommand{\Zp}{\mathbb{Z}_p}

\newcommand{\NN}{\mathbb{N}}

\definecolor{Green}{rgb}{0.0, 0.5, 0.0}

\newcommand{\cO}{\mathcal{O}}

\newcommand{\Q}{\mathbb{Q}}
\newcommand{\F}{\mathbb{F}}
\newcommand{\Z}{\mathbb{Z}}

\newcommand{\bG}{\mathbf{G}}

\newcommand{\Epn}{E^{(p^n)}}
\newcommand{\Eppn}{E'^{(p^n)}}

\newcommand{\EE}{\mathbb{E}}
\newcommand{\cV}{\mathcal{V}}
\newcommand{\VV}{\mathbb{V}}
\newcommand{\Spec}{\mathrm{Spec}}
\newcommand{\fY}{\mathfrak{Y}}
\newcommand{\fG}{\mathfrak{G}}
\newcommand{\fP}{\mathfrak{P}}

  \DeclareFontFamily{U}{wncy}{}
  \DeclareFontShape{U}{wncy}{m}{n}{<->wncyr10}{}
  \DeclareSymbolFont{mcy}{U}{wncy}{m}{n}
  \DeclareMathSymbol{\sha}{\mathord}{mcy}{"58}
  \DeclareMathSymbol{\zhe}{\mathord}{mcy}{"11}

  \renewcommand{\r}{\mathfrak r}
\renewcommand{\s}{\mathfrak s}
\renewcommand{\t}{\mathfrak t}
\newcommand{\cc}{\mathfrak c}
\newcommand{\fg}{\mathfrak{g}}
\newcommand{\fb}{\mathfrak{b}}
\newcommand{\cY}{\mathcal{Y}}
\title[Isogeny graphs and the Verschiebung map]{Isogeny graphs with level structures arising from the Verschiebung map}

\makeatletter
\let\@wraptoccontribs\wraptoccontribs
\makeatother

\usepackage{tikz,xcolor,hyperref}

\mathcode`l="8000
\begingroup
\makeatletter
\lccode`\~=`\l
\DeclareMathSymbol{\lsb@l}{\mathalpha}{letters}{`l}
\lowercase{\gdef~{\ifnum\the\mathgroup=\m@ne \ell \else \lsb@l \fi}}%
\endgroup

\author[A. Lei]{Antonio Lei}
\address[Lei]{Department of Mathematics and Statistics\\University of Ottawa\\
150 Louis-Pasteur Pvt\\
Ottawa, ON\\
Canada K1N 6N5}
\email{antonio.lei@uottawa.ca}

\author[K. Müller]{Katharina Müller}
\address[Müller]{Institut f\"ur theorethische Informatik, Mathematik und Operations Research\\
Universit\"at der Bundeswehr M\"unchen\\
Werner-Heisenberg-Weg 39\\
85577 Neubiberg\\
Germany}
\email{katharina.mueller@unibw.de}


\begin{document}
\begin{abstract}
We enhance an isogeny graph of elliptic curves by incorporating level structures defined by bases of the kernels of iterates of the Verschiebung map. We extend several previous results on isogeny graphs with level structures defined by geometric points to these graphs. Firstly, we prove that these graphs form $\mathbb{Z}_p$-towers of graph coverings as the power of the Verschiebung map varies. Secondly, we prove that the connected components of these graphs display a volcanic structure.
\end{abstract}

\maketitle

\section{Introduction}

Throughout this article, we fix a prime number $p$. Let $X$ be a finite connected graph. Vallières and McGown--Vallières \cite{vallieres,vallieres2,vallieres3} initiated the study of the Iwasawa theory of graph coverings by establishing an asymptotic formula for the $p$-part of the number of spanning trees in a {$\Zp$-covering of $X$, where $\Zp$ denotes the ring of $p$-adic integers}. This formula is analogous to the classical result of Iwasawa \cite{iwasawa73} on class numbers inside a $\Zp$-extension of a number field. Several authors have proven various results in the context of graph coverings, which resemble their counterparts in classical Iwasawa theory. These results can be found in the works of Gonet \cite{gonet-thesis,gonet22}, Lei--Vallières \cite{leivallieres}, Ray--Vallières \cite{anwesh-daniel}, Kleine--Müller \cite{KM,KM1}, Kataoka \cite{kataoka}, and Pengo--Vallières \cite{PV}.

Isogeny graphs refer to graphs, where the vertices represent isomorphism classes of elliptic curves defined over a finite field of characteristic $p$, and the edges represent $l$-isogenies between two elliptic curves, where $l$ is a prime number different from $p$. When the vertices represent supersingular elliptic curves, such graphs have important applications in cryptography; see \cite{CLG,DJP,EKHLMP}. In recent years, there has been a growing interest in understanding the structure of isogeny graphs enhanced with level structures (given by subgroups or torsion points of the elliptic curves). {For example, Arpin \cite{arpin} and Roda \cite{thesis-roda} studied properties of isogeny graphs of supersingular elliptic curves enhanced with $\Gamma_0(N)$-level and $\Gamma(N)$-level structures, respectively. Goren--Kassaei \cite{gorenkassaei} studied similar questions for $\Gamma_1(N)$-level structures in both ordinary and supersingular settings. Codogni-Lido recently developed a general framework for studying isogeny graphs with any level structure arising from a congruence subgroup in \cite{codogni-lido}. In \cite{LM1,LM2},} the authors of the present article investigated how isogeny graphs with level structures result in interesting towers of graph coverings resembling $p$-adic Lie extensions studied in the Iwasawa theory of number fields as one varies the levels. 

In \cite{LM1}, the following $\Zp$-towers were studied. Let $N\ge1$ be an integer coprime to two rational prime numbers $p$ and $l$. Let $k$ be a finite field of characteristic $p$. We define $G_N^m$ as the graph whose vertices consist of isomorphism classes of pairs $(E,P)$, where $E$ is an \textit{\textbf{ordinary}} elliptic curve defined over $k$ and $P$ is a point of $E(\overline{k})$ of order $Np^m$; {the edges of $G_N^m$ are given by $l$-isogenies}. There is a natural projection map $G_N^{m+1}\rightarrow G_N^m$ given by $(E,p)\mapsto (E,[p]P)$. The main reason we chose to work with ordinary elliptic curves was the lack of points of order $p$ on supersingular elliptic curves defined over a field of characteristic $p$. In \cite{LM2}, we circumvented this issue by considering supersingular elliptic curves defined over finite fields of characteristic $q$, equipped with level structures given by points on $E$ of order $p^nN$, where $q\nmid pN$.

In this article, we propose an alternative approach. Instead of defining isogeny graphs with level structures given by geometric points, we consider our elliptic curves as group schemes and furnish our graphs with level structures originating from kernels of iterates of the Verschiebung map. This allows us to consider both \textbf{\textit{ordinary \underline{and} supersingular}} cases simultaneously.

From now on, $k$ denotes a positive integer, and we write $S$ for the $\F_p$-scheme $\Spec(\FF_{p^k})$. We fix a prime number $l\ne p$ and a non-negative integer $N$ coprime to $p$ and $l$. For an integer $n\ge0$, we define in Definition~\ref{def:graph} a directed graph $G_l^p(n,N)$, whose set of vertices is given by isomorphism classes of $(E,Q,P)$, where $E/\FF_{p^k}$ is an elliptic curve, $Q$ is a geometric point on $E$ of order $N$ and $P$ is a generator of $\ker(V_E^n)$ ($V_E$ is the Verschiebung map of $E$). The set of edges is given by $l$-isogenies between such triples. {Note that the vertices of $G_\ell^p(0,1)$ are simply isomorphism classes of elliptic curves without any level structure.} {The undirected graph underlying} $G_l^p(n,N)$ is denoted by $\bG_l^p(n,N)$.

We prove the following generalization of \cite[Theorem~A]{LM1}.
\begin{lthm}\label{thmA}
     Let $E/\FF_{p^k}$ be an elliptic curve representing a non-isolated vertex of $\bG_l^p(0,1)$. Let $Q$ be a point on $E$ of order $N$. There is an integer $n_0$ such that for all $n\ge n_0$, we can find a connected component $X_n$ of $\bG_l^p(n,N)$ that includes a vertex of the form $(E,Q,P)$ for some $P$ and that the graphs $(X_n)_{n\ge n_0}$ form a $\Zp$-tower of graph coverings (in the sense of Definition~\ref{def:tower}).
\end{lthm}
\begin{remark}
    A vertex $E$ of $\bG_\ell^p(0,1)$ is isolated if and only if for {each $\ell$-isogeny $\phi:E\rightarrow E'$ the curve $E'$ is} not defined over $\mathbb{F}_{p^k}$.
\end{remark}

The structure of isogeny graphs consisting of ordinary elliptic curves (without level structures) was studied by Kohel \cite{kohel}. In particular, he proved that these graphs can be described as unions of the so-called volcano graphs. Colò--Kohel \cite{colokohel} introduced the notion of orientation on supersingular elliptic curves. Isogeny graphs defined by oriented supersingular elliptic curves exhibit a similar volcanic structure (see \cite{arpin-win,arpin-et-all,onuki,XZQ}). In \cite{LM2}, we describe the structure of the connected components of an isogeny graph of oriented supersingular elliptic curves enhanced with level structures. See Theorems~6.21, 6.22 and 6.23 of \textit{op. cit.} In the present article, we generalize these results to isogeny graphs whose level structures are given by kernels of iterates of the Verschiebung map, {which we denote by $\cG_l^p(n,N)$ (see Definition~\ref{def:oriented-graph}). To each connected component of $\cG_l^p(n,N)$, we attach an imaginary quadratic field (called the CM field of the connected component). We prove:
\begin{lthm}[{Theorems~\ref{thm:structure1}, \ref{thm:structure2} and \ref{thm:structure3}}]\label{thmB}
Let $\mathcal{Y}_n$ be a connected component of $\mathcal{G}_l^p(n,N)$ and let $K$ denote the CM field of $\cY_n$. 
\begin{itemize}
\item If $l$ splits in $K$, then $\cY_n$ is an undirected tectonic $l$-volcano.
\item If $l$ is ramified in $K$, then $\mathcal{Y}_n$ is an undirected $l$-volcano and the crater is a cycle graph.
    \item If $l$ is inert in $K$, then $\mathcal{Y}_n$ is an undirected $l$-volcano and the crater is disconnected.
    \end{itemize}
\end{lthm}
See Definitions~\ref{def:crater} and \ref{def:volcano} for the terminologies used in the statement of Theorem~\ref{thmB}.}

In Appendix \ref{app}, we study an inverse problem of tectonic craters relating them to isogeny graphs of oriented supersingular elliptic curves. This partially generalizes previous works of Bambury--Campagna--Pazuki \cite{pazuki} and the current authors \cite{LM1} in the ordinary case. While it is possible to include level structures coming from geometric points, our techniques do not extend to level structures arising from the Verschiebung map.

\subsection*{Acknowledgement}
The authors are indebted to the anonymous referee for numerous comments and suggestions on an earlier version of the article, which led to many improvements. We thank Hugo Chapdelaine, Daniel Larsson, and Riccardo Pengo for interesting discussions during the preparation of this article. AL's research is supported by the NSERC Discovery Grants Program RGPIN-2020-04259 and RGPAS-2020-00096.

\section{Notation}\label{S:notation}
Throughout this article, $E$ always denotes an elliptic curve over $\FF_{p^k}$. Let $\pi:E\to S=\Spec(\FF_{p^k})$ be the structure map. We write $F_E:E\to E^{(p)}$ and $V_E:E^{(p)}\to E$ for the Frobenius and Verschiebung maps attached to $E$, respectively. Here, $E^{(p)}$ is the pull-back of $\pi$ and $F_{{S}}$, where $F_S:S\rightarrow S$ denotes the absolute Frobenius (on affine rings, this corresponds to the endomorphism $x\mapsto x^p$). Note that $F_E$ and $V_E$ are dual isogenies of each other. The reader is referred to \cite[Chapter~12]{katz-mazur} for a more detailed discussion on these maps.

For a positive integer $n$, we write $F_E^n$ for the composition of Frobenii
\[
E\to E^{(p)}\to \cdots\to \Epn
\]
and define $V_E^n:\Epn\to E$ similarly. {Note that the kernel of $V_E^n$ is cyclic of order $p^n$.}
This allows us to identify $\Epn$ with $E/\ker(F_E^n)$ and $E$ with $\Epn/\ker(V_E^n)$.
When confusion may not arise, we remove the subscript $E$ from the notation {writing $F$ instead of $F_E$ and $V$ instead of $V_E$}.

\begin{remark}\label{rk:commute}
   Both $F$ and $V$ commute with isogenies. 
In particular, if $\phi\colon E\to E'$ is an isogeny of elliptic curves, $\phi$ induces a well-defined isogeny $\phi_n:\Epn\to\Eppn$ that maps $\ker(V_E^n)$ to $\ker(V_{E'}^n)$. When no confusion may arise, we simply write $\phi$ for $\phi_n$.
\end{remark} 

Given a graph $X$ (that is directed or undirected), we write $\VV(X)$ and $\EE(X)$ for the set of vertices and edges of $X$, respectively. In this paper, we allow for multiple edges and loops. This is crucial for our purposes, since there can be more than one isogeny relating two (equivalence classes of) elliptic curves.

\begin{defn}
A morphism of graphs $\phi\colon Y\to X$ consists of two maps $ \VV(Y)\to \VV(X)$ and $\EE(Y)\to \EE(X)$ such that if $e$ is an edge connecting $v$ and $w$, then $\phi(e)$ is an edge connecting $\phi(v)$ and $\phi(w)$. 
We call a surjective morphism of graphs $\pi\colon Y\to X$ a covering of graphs if $\pi$ is locally a bijection.

Let $Y/X$ be a covering of finite directed graphs with the projection map $\pi:Y\to X$. We say that $Y/X$ is a \textbf{$d$-sheeted covering} if each element of $\VV(X)$ has $d$ pre-images in $\VV(Y)$ under $\pi$.

The group of \textbf{deck transformations} of $Y/X$, denoted by $\Deck(Y/X)$, is the group of graph automorphisms $\sigma:Y\to Y$ such that $\pi\circ\sigma=\pi$.

We say that a $d$-sheeted covering $Y/X$ is \textbf{Galois} if $d=|\Deck(Y/X)|$. In this case, $\Deck(Y/X)$ is called the Galois group of the covering and is denoted by $\Gal(Y/X)$.
\end{defn}

We shall call the map that sends a directed graph to an undirected graph by ignoring the directions of the edges the \textit{\textbf{forgetful map}}. Note that we can regard any undirected graph as a directed one after replacing each edge by two directed edges with opposite directions {(see \cite[\S I.2.1]{serre})}.


\section{Defining isogeny graphs with level structures}
We introduce our definition of isogeny graphs with level structures arising from the Verschiebung map. Similarly to \cite[\S2]{LM1}, we analyze the number of connected components as the level varies, which constitutes the main ingredient of the proof of Theorem~\ref{thmA}. Recall that $k\ge1$ is an integer, $p$ and $l$ are fixed distinct prime numbers, and $N\ge 1$ is an integer that is coprime to $p$ and $l$.
\begin{defn}\label{def:graph}
\label{basic}
Let $n\ge0$ be an integer.
    \begin{itemize}
        \item[i)] Let $E$ and $E'$ be elliptic curves defined over $\FF_{p^k}$. Let $Q$ and $Q'$ be geometric points of order $N$ on $E$ and $E'$, respectively. Let $P_n$ and $P_n'$ be generators of $\ker(V_E^n\colon \Epn\to E)$ and $\ker (V_{E'}^n\colon \Eppn\to E')$, respectively. We say that tuples $(E,Q,P_n)$ and $(E',Q',P_n')$ are equivalent if there is an isomorphism ${\psi}\colon E\to E'$ such that ${\psi}(Q)=Q'$ and ${\psi}(P_n)=P_n'$. {By an abuse of notation, we shall denote the equivalence class containing the triple $(E,Q,P_n)$ by the same symbol.}
                \item[ii)] Let $E$ and $E'$ be elliptic curves defined over $\FF_{p^k}$. Let $\phi,\phi':E\to E'$ be isogenies. We say that $\phi$ and $\phi'$ are equivalent if they have the same kernel. 
        \item[iii)]   Let $G_l^p(n,N)$ be the directed graph given by:
    \begin{itemize}
        \item The vertices of $G_l^p(n,N)$ are the equivalence classes of tuples $(E,Q,P_n)$ given as in  i).
    \item The directed edges from $(E,Q,P_n)$ to $(E',Q',P_n')$ are the equivalence classes of $l$-isogenies $\phi\colon E\to E'$ such that $\phi(Q)=Q'$ and {$\phi_n(P_n)=P_n'$, where $\phi_n:E^{(p^n)}\rightarrow (E')^{(p^n)}$ is the isogeny discussed in Remark~\ref{rk:commute}}.
    \end{itemize}
    \end{itemize}
\end{defn}
If no confusion may arise, we shall drop the subscript $n$ and simply write $P$ for $P_n$. When $n=0$ and $N=1$, we may drop $P_n$ and $Q$ from the notation and simply write $E$ for a vertex of the graph $G_l^p(0,1)$. We write $\bG_l^p(n,N)$ for the image of $G_l^p(n,N)$ under the forgetful map.
\begin{lemma}\label{lem:cover}
    Let $n\ge1$ be an integer. The map $(E,Q,P)\mapsto (E,Q,V(P))$ induces a {degree-$p$} graph covering
    \[\pi_{n,n-1}\colon G_l^p(n,N)\to G_l^p(n-1,N).\]
\end{lemma}
\begin{proof}
   It follows from \cite[Proposition 12.3.2(3)]{katz-mazur} that $P\in \Epn({\FF_{p^k}})$ generates $\ker(V^n)$ if and only if $V(P)$ generates $\ker(V^{n-1})$. Therefore, $\pi_{n,n-1}(\VV(G_l^p(n,N)))=\VV(G_l^p(n-1,N))$. To show that $\pi_{n,n-1}$ is a morphism of graphs, we need to show that if there is an edge  $e\in\EE(G_l^p(n,N))$ linking $(E,Q,P)$ to $(E',Q',P')$, then there is an edge in $G_l^p(n-1,N)$ linking $(E,Q,V(P))$ to $(E',Q',V(P'))$.
   
    Let $\phi\colon E\to E'$ be an $l$-isogeny corresponding to $e$. Then $\phi(Q)=Q'$ and $\phi(P)=P'$.  As \[\phi(V(P))=V(\phi(P))=V(P'),\] the isogeny $\phi$ gives rise to an edge linking $(E,Q,V(P))$ to $(E',Q',V(P'))$. 
    Therefore, $\pi_{n,n-1}$ is indeed a morphism of graphs. 
    
    It remains to show that $\pi_{n,n-1}$ is a graph covering. Let $v\in \VV(G_\ell^p(n,N))$. We have to show that the edges starting (resp. ending) at $v$ are in one-to-one correspondence with the edges starting (resp. ending) at $\pi_{n,n-1}(v)$. Each edge is generated by a unique {equivalence} class of $\ell$-isogenies. The number of isomorphism classes of these isogenies depends only on the elliptic curve $E$, not on the level structure. Therefore, $\pi_{n,n-1}$ is locally {a} bijection, as desired. 
\end{proof}
\begin{corollary}\label{cor:cover}
    Let $n\ge m\ge 0$ be integers. The composition 
    \[\pi_{n,m}=\pi_{n,n-1}\circ \pi_{n-1,n-2}\circ \dots \circ \pi_{m+1,m}:G_l^p(n,N)\to G_l^p(m,N)\]
    is a graph covering.
\end{corollary}
\begin{proof}
    This follows immediately from Lemma~\ref{lem:cover} since the composition of graph coverings is a graph covering.
\end{proof}
\begin{lemma}
\label{connected-components}
Let $\mathcal{X}$ be a connected component of $G_l^p(0,1)$ that consists not only of an isolated vertex. Let $\cX(n,N)$ be the subgraph of $G_l^p(n,N)$ whose vertices are of the form $(E,Q,P_n)$ where $E$ is a vertex of $\mathcal{X}$. 
    The number of connected components in $\cX(n,N)$ is uniformly bounded in $n$.
\end{lemma}
\begin{proof}
    By definition, $\cX(n,N)$ is a finite graph. Let $E$ be a representative of a vertex in $\cX$, and let $Q$ be a fixed geometric point in $E$ of order $N$. We define a relation $\sim$ in $\VV(\cX(n,N))$ as follows: $(E,Q,P)\sim (E',Q',P')$ if these two vertices are in the same connected component in $\cX(n,N)$. It follows from the discussion in \cite[Remark~2.4]{LM1} that if $\phi$ is an $l$ -isogeny that induces an edge in $\cX(n,N)$ from $(E,Q,P)$ to $(E',Q',P')$, then the dual isogeny $\hat\phi$ induces an edge in $\cX(n,N)$ from $(E',Q',P')$ to {the triple $(lE,[l]Q,[l]P)$, which is equivalent to the triple }$(E,[l]Q,[l]P)$. This implies that $\sim$ is an equivalence relation.

    We show that the number of equivalence classes of $\sim$ is uniformly bounded in $n$. Let $r$ be an integer such that {$\ell^r\equiv 1\pmod N$; such an integer exists as $\ell$ is coprime to $N$. We have} $[l]^rQ=Q$. Then all vertices of the form $(E,Q,[l]^{rj} P)$, $j\in\NN$, are in the same equivalence class. Let $U_n\subset (\ZZ/p^n\ZZ)^\times $ be the subgroup generated by $l^r$. Recall that $\Z_p^\times\cong \mu_q\times (1+2p\Z_p)$, where $q=4$ if $p=2$ and $q=p-1$ otherwise. Here, $\mu_q$ denotes the group of $q$-th roots of unity. In particular, the index of $U_n$ in $(\ZZ/p^n\ZZ)^\times$ is uniformly bounded {by the index of $\overline{\langle l^r\rangle}$ in $\Zp^\times$, which is finite since the image of $l^r$ in $1+2p\Z_p$ is non-trivial}.

    By \cite[Lemma 12.2.3]{katz-mazur}, $V^n$ is a cyclic $p^n$-isogeny. In particular, the automorphism group of $\ker(V^n)$ is isomorphic to $(\ZZ/p^n\ZZ)^\times $, which acts transitively on the set of generators of $\ker(V^n)$. As the index of $U_n$ in $(\ZZ/p^n)^\times $ is uniformly bounded, we see that the number of equivalence classes under $\sim$ is also uniformly bounded, which concludes the proof of the lemma.
\end{proof}

We conclude this section by showing that, in the ordinary case, we recover the isogeny graphs defined in \cite{LM1}.
\begin{defn}
Define $H_l^p(n,N)$ as the maximal subgraph of $\mathbf{G}_l^p(n,N)$ whose vertices are of the form $(E,Q,P)$, where $E$ is an ordinary elliptic curve. 
\end{defn}
Let $G_N^m$ be the undirected graph defined in \cite[Definition~2.1]{LM1}. The set of vertices of $G_N^m$ consists of equivalence classes of tuples $(E,P)$ of ordinary elliptic curves $E/\FF_{p^k}$ and geometric points $P$ in $E(\overline{\FF}_p)$ of order $Np^m$; {as $N$ is coprime to $p^m$, we can write a tuple $(E,P)$ as $(E,Q+P')$, where $Q$ is of order $N$ and $P'$ is of order $p^m$. }The edges of $G_N^m$ are given by $l$-isogenies. Clearly, we have the following isomorphism of graphs
\[H_l^p(0,N)\cong G_N^0.\]
\begin{proposition}
    For each integer $n\ge 0$, there is an isomorphism of graphs $H_l^p(n,N)\cong G_N^n$.
\end{proposition}
\begin{proof}
Let $E/\FF_{p^k}$ be an ordinary elliptic curve. By \cite[Proposition~12.3.6]{katz-mazur}, $V^n$ is étale. Let $(E,Q+P)$ be a vertex of $G_N^n$, where $P\in E[p^n]$ and $Q\in E[N]$. Proposition 12.2.7 in op. cit. tells us that there exist an elliptic curve $E_n/\FF_{p^k}$ and an isomorphism $\psi\colon E\to E_n^{(p^n)}$ such that $\psi(P)$ is a generator of $\ker(V_{E_n}^n)$. Furthermore, since $N$ is coprime to $p$, $V_{E_n}^n\circ\psi(Q)$ is a point on $E_n$ of order $N$. Therefore, we have a well-defined map
    \[{\alpha}:\VV(G_N^n)\to \VV(H_l^p(n,N)),\quad (E,Q+P)\mapsto (E_n,V_{E_n}^n\circ\psi(Q),\psi(P)) .\]
    It is clear that ${\alpha}$ is a bijection. {Indeed, there is a one-to-one correspondence between pairs $(E,P)$ and pairs $(E_n,\psi(P))$. As $V^n_{E_n}$ induces a group {isomorphism} $E^{(p^n)}_n[N]\stackrel\sim\to E_n[N]$, $\alpha$ admits an inverse given by
    \[(E_n,Q',P')\mapsto\left(E,\psi^{-1}\circ \left(V_{E_n}^{(p^n)}\right)^{-1}(Q'),\psi^{-1}(P')\right).\]}

    It remains to show that ${\alpha}$ preserves the edges of the two graphs. Let $\phi\colon E\to E'$ be an $l$-isogeny, and let $\psi\colon E\to E_n^{(p^n)}$ and $\psi'\colon E'\to {E'_n}^{(p^n)}$ be the isomorphisms given by \emph{loc. cit.}. Then $\phi_n=\psi'\circ \phi\circ \psi^{-1}\colon E^{(p^n)}_n\to {E'}^{(p^n)}_n$ is an $l$-isogeny, which in turn induces an $l$-isogeny from $E_n\to {E_n'}$ by Remark~\ref{rk:commute}. One can directly check that ${\alpha}$ commutes with $\phi\mapsto \phi_n$, as required.
\end{proof}

\section{Voltage assignments and derived graphs}
\label{sec:voltage}

\subsection{General results on voltage assignments and derived graphs}
We recall the notion of voltage assignments and their derived graphs.

\begin{defn}
    Let $X$ be a directed graph and $(G,\cdot)$ a group. A $G$-valued \textbf{voltage assignment} on $X$ is a function $\alpha:\EE(X)\rightarrow G$. 

    Given a $G$-valued voltage assignment on $X$, we define the derived graph $X(G,\alpha)$ whose vertices and edges are given by $\VV(X)\times G$ and $\EE(X)\times G$, respectively. {If $e\in\EE(X)$ is an edge of $X$ going from $s$ to $t$, then the edge $(e,\sigma)\in\EE(X(G,\alpha))= \EE(X)\times G$ goes from $(s,\sigma)$ to $(t,\sigma\cdot\alpha(e))$.}

If $X$ is an undirected graph, we define $\tilde X$ to be the directed graph, where $\VV(X)=\VV(\tilde X)$ and $\EE(\tilde X)$ is obtained from replacing each edge in $\EE(X)$ by two edges with opposite directions. We say that $\alpha$ is a voltage assignment on $X$ if it is a voltage assignment on $\tilde X$ that is compatible with the inversion of edges. We denote by $X(G,\alpha)$ the undirected graph obtained from the directed graph $\tilde X(G,\alpha)$ after identifying two edges with opposite directions with a single undirected edge.
\end{defn}

The derived graph $X(G,\alpha)$ is naturally a covering of $X$ (via $(x,\sigma)\mapsto x$). The following theorem gives a criterion on this covering being Galois:

\begin{theorem}\label{thm:Gonet}
    Let $X$ be a finite connected undirected graph, $G$ a finite group, and $\alpha$ a $G$-valued voltage assignment on $X$. If $X(G,\alpha)$ is connected, then $X(G,\alpha)$ is a Galois covering of $X$. Furthermore, its Galois group is isomorphic to $G$. {Conversely, if $Y/X$ is a Galois covering of finite connected graphs, then $Y$ is isomorphic to the derived graph of a $\Gal(Y/X)$-valued voltage assigment on $X$.}
\end{theorem}
\begin{proof}
See \cite[Theorem~2.10]{gonet22} and \cite[Theorem 8]{gonet-thesis} (note that the proof of this theorem is based on the work of Gross--Tucker in \cite[Chapter~2]{grosstucker}).
\end{proof}
Later in this paper, we shall employ the following generalization of Theorem~\ref{thm:Gonet}:

 \begin{lemma}
    \label{connected component galois}
        Let $X$ be a finite connected undirected graph and let $\alpha$ be a voltage assignment on $X$ with values in a finite abelian group $G$. Let $Z$ be a connected component of $X(\alpha,G)$. {Let $H$ be the stabilizer of $Z$ under the action of $G$ on connected components. Then $Z$ is isomorphic to $X(H,\alpha)$ and} $Z/X$ is a Galois covering whose Galois group is isomorphic to $H$.
    \end{lemma}
    \begin{proof}
        The group $G$ acts on $\VV(X(\alpha,G))$ via $h\cdot (v,g)=(v,hg)$. This induces an action of $G$ on the set of connected components of $X(\alpha,G)$. Let $H$ be the stabilizer of $Z$ under this action.

        Recall from \cite[Proposition 2.4i) and the proof of Corollary 2.5]{LM2} that $G$ acts transitively on the connected components of $X(\alpha,G)$ and the connected components are isomorphic to each other. This tells us that $X(\alpha,G)$ has $|G|/|H|$ connected components, and each vertex of $\VV(X)$ has $|H|$ pre-images in $Z$. Since $H$ stabilizes $Z$, for each $w\in\VV(X)$, its pre-image in $Z$ is given by $\{(w,g_wh)\mid h\in H\}$ for some $g_w\in G$.
        
         We define a voltage assignment
        \[\beta \colon \EE(\tilde X)\to H, \quad e\mapsto \alpha(e)g_wg_{w'}^{-1},\]
        if $e$ is an edge from $w$ to $w'$. There is a natural bijective map
        \[\VV(Z)\to \VV(X(\beta_n,H)), \quad (w,g_wh)\mapsto (w,h).\]
        It remains to check that this map respects the edges. There is an edge from $(w,g_wh)$ to $(w',g_{w'}h')$ in $Z$ if and only if there is an edge $e$ from $w$ to $w'$ in $X$ and $\alpha(e)g_wh=g_{w'}h'$. This is equivalent to $\beta(e)h=h'$, which shows that there is a graph isomorphism
        \[Z\cong X(\beta,H).\]
        As $Z$ is connected, it follows from Theorem~\ref{thm:Gonet} that $Z/X$ is a Galois covering, whose Galois group is isomorphic to $H$.
        \end{proof}

We recall the following definition of $\Zp$-towers of graph coverings.

\begin{defn}\label{def:tower}
Let $X$ be a finite connected undirected graph. A \textbf{$\Zp$-tower} over $X$ is a sequence of graph coverings
\[
X=X_0\leftarrow X_1\leftarrow X_2\leftarrow \cdots \leftarrow X_n\leftarrow\cdots
\]
such that for each $n\ge0$, the covering $X_n/X$ is Galois, whose Galois group is isomorphic to $\ZZ/p^n\ZZ$.
\end{defn}

This coincides with the definition of "abelian $p$-towers" introduced in \cite{vallieres,vallieres2,vallieres3}. Let $\alpha$ be  $\Zp$-valued voltage assignment on a finite connected undirected graph $X$. For an integer $n\ge 1$, define $\alpha_n$ as the composition $$\EE(\tilde X)\stackrel\alpha\to \Zp\to\ZZ/p^n\ZZ, $$
where the second map is the natural projection map. A characterization for the connectedness of the derived graphs $X(\alpha_n,\ZZ/p^n\ZZ)$ has been established in \cite[Theorem~2.15]{DLRV}. When these graphs are connected, they form a $\Zp$-tower over $X$ by Theorem~\ref{thm:Gonet}. 

The following lemma shows how the voltage assignment can be modified to shift the base graph of a $\Zp$-tower to an intermediate covering.

\begin{lemma}
\label{lem:voltage}
    Let $X$ be a finite connected undirected graph and let $\alpha\colon \EE(\tilde X)\to \Z_p$ be a voltage assignment. For an integer $n\ge 0$, write $\alpha_n$ for the voltage assignment 
    \[\EE(\tilde X)\to \Z_p\to \ZZ/p^n\ZZ\]
    induced by $\alpha$, and write $ X_n= X(\alpha_n,\ZZ/p^n\ZZ)$.
    
    For a fixed $n\ge0$, there is a voltage assignment $\beta:\EE(\tilde X_n)\to p^n\Zp$ such that for all integers $m\ge n$, $X_m$ is isomorphic to the derived graph $ X_n(\beta_m,\ZZ/p^{m-n}\ZZ)$, where $\beta_m$ is defined by
    \[{\beta_m\colon}\EE(\tilde  X_n)\stackrel{\beta}{\to}p^n\Zp\to p^n\Zp/p^m\Zp\cong \ZZ/p^{m-n}\ZZ.\] \end{lemma}
    \begin{proof}
        Recall that $\EE( \tilde X_n)$ consists of tuples $(e,\sigma)$, where $\sigma\in\ZZ/p^n\ZZ$ and $e\in\EE(\tilde X)$. {Fix a set-theoretic section $\tau\colon \ZZ/p^n\ZZ\to \Z_p$ of the projection $\Zp\rightarrow\ZZ/p^n\ZZ$ and denote $\tau(\sigma)$ by $\tau_\sigma$. }
        We define the voltage assignment $\beta$ as follows
        \[\beta\colon\EE(\tilde  X_n)=\EE( \tilde X)\times \ZZ/p^n\ZZ\to p^n\Z_p,\quad (e,\sigma)\mapsto \alpha(e)+\tau_{\sigma}-\tau_{\sigma'},\]
        where $\sigma'=\sigma+\alpha_n(e)$. {The map $\beta$ is well-defined because} the image of $\alpha(e)+\tau_\sigma-\tau_{\sigma'}$ in $\ZZ/p^n\ZZ$ is equal to $\alpha_n(e)+\sigma-{\sigma'}=0$. Thus, $\alpha(e)+\tau_\sigma-\tau_{\sigma'}$ lies in $p^n\ZZ_p$. 
        
        For $m\ge n$, let $\beta_m$ be the composition
        \[{\beta_m\colon }\EE(\tilde X_n)\stackrel\beta\to p^n\Z_p\to p^n\ZZ/p^m\ZZ.\]
        Let $H_m=p^n\ZZ/p^m\ZZ$ and $G_m=\ZZ/p^m\ZZ$. Then $H_m$ is a normal subgroup of $G_m$ and $G_m/H_m\cong\ZZ/p^n\ZZ$.
         Let $\bar\tau_\sigma$ denote the image of $\tau_\sigma$ in $G_m$ under the natural projection map. Then $\{\bar\tau_\sigma:\sigma\in\ZZ/p^n\ZZ\}$ is a set of coset representations of $G_m/H_m$. In particular, every element $a\in G_m$ can be expressed as $a=a-\bar\tau_{\sigma_a}+\bar\tau_{\sigma_a}$, where $\sigma_a\in\ZZ/p^n\ZZ$ is the image of $a$ under the natural projection map and $a-\bar\tau_{\sigma_a}\in H_m$.
         
        Recall that $\VV( X_m)=\VV( X)\times G_m$. Define
        \begin{align*}
        {\varepsilon_m}:\VV( X_m)&\rightarrow \VV(X)\times\ZZ/p^n\ZZ\times p^n\ZZ/p^m\ZZ=\VV(X_n)\times H_m,\\
        (v,a)&\mapsto(v,\sigma_a,a-\bar\tau_{\sigma_a}),
        \end{align*}
        where $v\in\VV(X)$, $a\in\ZZ/p^m\ZZ$. Clearly, ${\varepsilon_m}$ is a bijection. Furthermore, there is an edge in $\EE(\tilde X_m)$ from $(v,a)$ to $(w,b)$ if and only if there is an edge $e\in\EE(\tilde X)$ from $v$ to $w$ and $b=a+\alpha_m(e)$, which means that
        \[
        \sigma_b=\sigma_a+\alpha_n(e),\quad  \beta_m((e,\sigma_a))=(b-\bar\tau_{\sigma_b})-(a-\bar\tau_{\sigma_a}).
        \]
        Note that the first equation signifies that the edge $(e,\sigma_a)\in\EE(\tilde X_n)=\EE(\tilde X)\times\ZZ/p^n\ZZ$ links $(v,\sigma_a)$ to $(w,\sigma_b)$, whereas the second equation signifies that the edge $((e,\sigma_a),a-\overline\tau_{\sigma_a})$ in $\EE(\tilde X_n(\beta_m,H_m))$ links $((v,\sigma_a),a-\bar\tau_{\sigma_a})$ to $((w,\sigma_b),b-\bar\tau_{\sigma_b})$.
        Thus, $X(\alpha_m,\ZZ/p^m\ZZ)\cong  X_n(\beta_m,H_m).$
    \end{proof}

    \begin{remark}
The proof of Lemma~\ref{lem:voltage} may already be well-known to experts. In fact, it can be adapted for more general settings, replacing $\Zp$ by a general group $G$, and the subgroup $p^n\Zp$ by a normal subgroup $H$ of $G$.\end{remark}

Inspired by our previous work \cite{LM1} on isogeny graphs arising from ordinary elliptic curves, we prove how $\Zp$-valued voltage assignments give rise to $\Zp$-towers without assuming connectivity. Instead, we assume that the number of connected components is uniformly bounded and show that we obtain a collection of $\Zp$-towers. {Note that if $Y/X$ is a graph covering, the number of connected components of $Y$ is larger than or equal to that of $X$. Thus, along a $\Z_p$-tower
\[
X=X_0\leftarrow X_1\leftarrow X_2\leftarrow \cdots \leftarrow X_n\leftarrow\cdots,
\]
 the number of connected components of $X_n$ is a non-decreasing function in $n$. }
    
    \begin{corollary}
    \label{cor:voltage-assignment}
        Let $X$ be a finite connected undirected graph, and let $\alpha$ be a $\Zp$-valued voltage assignment on $X$. Let $\alpha_n$ be as before. Assume that there exists an integer $m_0$ such that the number of connected components of $X_n:=X(\alpha_n,\ZZ/p^n\ZZ)$ is equal to that of $X_{m_0}$ for all $n\ge m_0$. 
        
        Let $\mathcal{G}_{m_0}$ be a connected component of $X_{m_0}$. For each $m\ge m_0$, there exists a unique subgraph $\cG_m$ of $X_m$ such that
        \[\mathcal{G}_{m_0}\longleftarrow\mathcal{G}_{m_0+1}\longleftarrow\mathcal{G}_{m_0+2}\longleftarrow\dots\longleftarrow\dots\]
        is a $\Zp$-tower.
    \end{corollary}
    \begin{proof}For each $m$, there exists at least one connected component $\mathcal{G}_m$ that surjects onto $\mathcal{G}_{m_0}$. As the number of connected components in {$X_m$} is constant, there exists exactly one such connected component.
         By Lemmas \ref{connected component galois} and  \ref{lem:voltage}, there is a voltage assignment $\beta$ on $\mathcal{G}_{m_0}$ with values in $p^{m_0}\Z_p$ such that $\mathcal{G}_m=X(\beta_m,p^{m_0}\Z_p/p^m\Z_p)$. As $\mathcal{G}_m$ is connected, it follows that $\mathcal{G}_m/\mathcal{G}_{m_0}$ is indeed Galois with $\Gal(\mathcal{G}_m/\mathcal{G}_{m_0})\cong \ZZ/p^{m-m_0}{\ZZ}$ by Theorem~\ref{thm:Gonet}.
    \end{proof}

\subsection{Proof of Theorem~\ref{thmA}}
    
    Recall that  $\cX$ denotes a connected component of $G_\ell^p(0,1)$ and that $\cX(n,N)$ is the subgraph of $G_\ell^p(n,N)$ whose vertices are of the form $(E,Q,P_n)$ where $E$ is a vertex of $\cX$ (see Lemma~\ref{connected-components}). In light of Corollary \ref{cor:voltage-assignment}, in order to prove Theorem~\ref{thmA}, it suffices to show that there exists a $\Zp$-valued voltage assignment $\alpha$ on $\cX(0,N)$ such that $\cX(n,N)$ is the derived graph for the voltage assignment $\alpha_n$ given by $\alpha$ modulo $p^n$. We do so by following the strategy of \cite[Appendix A]{LM1}.

    For every elliptic curve $E/\FF_{p^k}$, {let $\cV(E)$ denote the inverse limit $\displaystyle\varprojlim_n\ker(V^n)$. It is a free $\Zp$-module of rank one. We fix a $\Zp$-basis $t_E=(P_n)_{n\ge1}$ of $\cV(E)$. In other words, $P_n$ is a generator of $\ker(V^n)$ for all $n\ge1$, and $V(P_n)=P_{n-1}$ for all $n\ge2$).}
        For each $n$, there is a group isomorphism
    \[\psi_{E,n}\colon \ker(V^n)\to \ZZ/p^n\ZZ,  \quad P\mapsto a_P, \]
where $a_P$ is the unique element of $\ZZ/p^n\ZZ$ such that $P=a_PP_n$.
     
    For each equivalence class of isogenies between elliptic curves, we fix one representative once and for all.  Let $\phi\colon E\to E'$ be an $l$-isogeny that induces an edge $(E,Q)\to (E',Q')$ in $\cX(0,N)$, which we denote by $e_\phi$. There is a unique element $t_\phi$ in $\Z_p$ such that 
    \[\phi_*(t_E)=t_\phi t_{E'},\]
    where $\phi_*\colon \cV(E)\to \cV(E')$ is the $\Zp$-homomorphism induced by $\phi$. Let $t_{\phi,n}$ denote the image of $t_\phi$ in $\ZZ/p^n$ under the natural projection map. Then
    \[\psi_{E,n}(P)t_{\phi,n}=\psi_{E',n}(\phi(P)).\]
 If $\alpha$ is the $\Zp$-valued voltage assignment on $\cX(0,N)$ given by 
 \[\alpha(e_\phi)=t_\phi,\]
the discussion above shows that $\cX(n,N)$ is isomorphic to $X(\alpha_n,\ZZ/p^n\ZZ)$, as required.

\begin{remark}
    We have chosen to vary only the power of the Verschiebung map in Theorem~\ref{thmA}. In light of the recent work of Pengo--Vallières \cite{PV} on "$\widehat\ZZ$-tower" of graph coverings, it can also be interesting to vary the level $N$ over integers that are coprime to $l$ and $p$. This can potentially result in a "$\displaystyle\prod_{q\nmid pl}\ZZ_q$-tower" of graph coverings.
\end{remark}

\section{Volcanic structures of isogeny graphs of oriented supersingular elliptic curves}

Recall that an elliptic curve $E/\FF_{p^k}$ is said to be \textit{supersingular} if $E(\overline\FF_{p^k})$ does not contain an element of order $p$.
Throughout this section, we assume that $p\equiv 1\pmod{12}$ and that $S=\Spec(\mathbb{F}_{p^2})$. Our goal is to study the volcanic structure of isogeny graphs arising from supersingular elliptic curves enhanced by orientations (first introduced in \cite{colokohel}) together with level structures. Unlike \cite{LM2} where we studied oriented supersingular elliptic curves with $\Gamma(M)$-level structure for some positive integer $M$ such that $p\nmid M$, we shall study level structures arising from the kernels of iterates of the Verschiebung map.

\subsection{Definitions}\label{S:def-volcano}
Throughout, $N$ is a positive integer that is coprime to $p$ and $l$ as before.
\begin{defn}\label{def:oriented-elliptic} Let $E/\FF_{p^2}$ be a supersingular elliptic curve and $K$ is an imaginary quadratic field such that there exists an embedding $\iota:K\hookrightarrow \End(E)\otimes \QQ$ ({such $\iota$ exists if and only if $p$ is not split in $K$}).
\item[i)] We call the pair $(E,\iota)$ an \textbf{oriented elliptic curve}. 
\item[ii)]     We call an order $\mathcal{O}$ of $K$ \textbf{primitive} for $(E,\iota)$, if $\iota(\mathcal{O})=\iota(K)\bigcap \textup{End}(E)$. 
\item[iii)] We say that $\mathcal{O}$ is $l$-\textbf{primitive} if $l$ does not divide $[\iota(K) \bigcap \textup{End}(E):\iota(\mathcal{O})]$.
\item[iv)] If $\phi\colon E\to E'$ is an isogeny of elliptic curves, we define the orientation $\phi_*\iota$ on $E'$ by
\begin{align*}
    \phi_*\iota:K&\rightarrow\End(E')\otimes \QQ,\\
    \alpha&\mapsto\frac{1}{\deg(\phi)}{\phi}\circ\iota(\alpha)\circ\widehat\phi,
\end{align*}
where $\widehat\phi$ is the dual isogeny of $\phi$.
\item[v)]    An \textbf{isogeny of oriented elliptic curves} $\phi\colon (E,\iota)\to (E',\iota')$ is an isogeny of elliptic curves $\phi\colon E\to E'$ such that $\phi_*\iota=\iota'$.
\item[vi)]  Let $Q$ be a geometric point of order $N$ of $E$ and let $P$ be a generator of $\ker(V^n)$.
The tuple $(E,\iota,Q,P)$ is called an \textbf{oriented elliptic curve with $\Gamma_1(Np^n)$-level structure}.
\item[vii)] Two oriented elliptic curves with $\Gamma_1(Np^n)$-level structures $(E,\iota,Q,P)$ and $(E',\iota',Q',P')$ are said to be \textbf{equivalent} if there exists an isomorphism of elliptic curves $\psi\colon E\to E'$ such that 
\begin{itemize}
\item $\iota'=\psi_*\iota$. 
\item $\psi(Q)=Q'$
\item $\psi(P)=P'$
\end{itemize}
\end{defn}

Let $\iota:K\hookrightarrow\End(E)\otimes K$ be an orientation. There is a unique order $\cO$ in $K$ such that $\iota(\cO)=\iota(K)\bigcap\End(E)$ (namely, $\iota^{-1}\left(\iota(K)\bigcap\End(E)\right)$. The order $\mathcal{O}$ is of the form $\ZZ\oplus \ZZ f \omega_K$, where $\ZZ\oplus \ZZ \omega_K$ is the ring of integers of $K$ and $f$ is a unique positive integer. The integer $f$ is called the \textbf{\textit{conductor}} of $\mathcal{O}$.

\begin{defn}\label{def:oriented-graph}
    We define $\mathfrak{G}_l^p(n,N)$ as the directed graph whose vertices are equivalence classes of oriented elliptic curves with $\Gamma_1(Np^n)$-level structure and whose edges are given by equivalence classes of $\ell$-isogenies {that respect both the orientation and the level structure, i.e., the edges from $(E,\iota,Q,P)$ to $(E',\iota',Q',P')$ correspond to $l$-isogenies $\phi:E\rightarrow E'$ such that $\iota'=\phi_*\iota$, $\phi(Q)=Q'$ and $\phi(P)=P'$}. The undirected graph obtained from $\mathfrak{G}_l^p(n,N)$ under the forgetful map is denoted by $\mathcal{G}_l^p(n,N)$.
\end{defn}

\begin{defn}
Let $\phi:(E,\iota)\rightarrow(E',\iota')$ be an isogeny of oriented supersingular elliptic curves.  Assume that $\mathcal{O}$ and $\cO'$ are orders of $K$ that are primitive for $(E,\iota)$ and $(E',\iota')$, respectively. We say that $\phi$ is \begin{itemize}
        \item {\bf horizontal} if $\mathcal{O}=\mathcal{O}'$,
        \item {\bf descending} if $\mathcal{O}'\subset \mathcal{O}$,
        \item {\bf ascending} if $\mathcal{O}'\supset \mathcal{O}$.
    \end{itemize}
\end{defn}

The notion of horizontal, ascending, and descending isogenies allows us to define the depth of an oriented elliptic curve.
\begin{defn}
    Let $(E,\iota)$ be an oriented elliptic curve. Assume that $\mathcal{O}$ is a primitive order with respect to $(E,\iota)$. We say that $(E,\iota)$ is of depth zero if the conductor of $\mathcal{O}$ is not divisible by $l$. We say that $(E,\iota)$ is of depth $r$ if there is a path in $\mathfrak{G}_l^p(0,1)$ consisting of $r$ descending $l$-isogenies that goes from a depth-zero oriented elliptic curve to $(E,\iota)$.
    
    Let $v=(E,\iota,Q,P)$ be a vertex of $\mathfrak{G}_l^p(n,N)$. We define the depth of $v$ as the depth of $(E,\iota)$. 
\end{defn}

\begin{defn}\label{def:partial}
    For integers $N\ge 1$ and $n\ge 0$, we define the following morphism of {directed} graphs:
    \[\partial\colon \mathfrak{G}_l^p(n,N)\to G_l^p(n,N),\quad (E,\iota,Q,P)\mapsto (E,Q,P),\] 
   which induces a morphism of undirected graphs
    \[\partial\colon \mathcal{G}_l^p(n,N)\to \mathbf{G}_l^p(n,N),\quad (E,\iota,Q,P)\mapsto (E,Q,P).\] 
\end{defn}

We now recall the following definitions first introduced in \cite[\S5]{LM1}:
\begin{defn}\label{def:crater-tec}
  Let $\r,\s,\t,\cc$ be nonnegative integers. We say that a directed graph is an \textbf{abstract tectonic crater} of parameters $(\r,\s,\t,\cc)$ if it satisfies
    \begin{itemize}
  \item[a)] There are $\r\s\t$ vertices;
  \item[b)] Each edge is assigned a color -- blue or green;
  \item[c)] At each vertex $v$, there is exactly one blue edge with $v$ as the source, and exactly one blue edge with $v$ as the target, and similarly for green edges;
  \item[d)] Starting at each vertex, there is exactly one closed blue (resp. green) path without backtracks of length $\r\s$ (resp. $\r\t$);
  \item[e)] After every $\s$ (resp. $\cc\t$) steps in the the closed blue (resp. green) path given in d), the two paths meet at a common vertex.  
  \end{itemize}
\end{defn}

It can be checked that an abstract tectonic crater is uniquely determined by its parameters, up to isomorphisms of graphs. Below are three examples of abstract tectonic craters with parameters $(4,3,1,1)$, $(3,2,2,1)$ and $(5,1,1,2)$, respectively. They are taken from \cite[Examples 4.29 and 4.22]{LM1} and \cite[Example~6.19]{LM2}.
\begin{center}
     \begin{tikzpicture}[scale=0.55, every node/.style={circle, draw, fill=black, minimum size=4pt, inner sep=0pt}]
\foreach \X[count=\Y] in {black,black,black,black,black,black,black,black,black,black,black,black}
{\node[draw,circle,\X] (x-\Y) at ({30*\Y+60}:3.5){}; }
\foreach \X[count=\Y] in {0,...,11}
{\ifnum\X=0
\draw[->-=.5,blue,line width=0.3mm] (x-12) --(x-\Y) ;
\else
\draw[->-=.5,blue,line width=0.3mm] (x-\X) --(x-\Y) ;
\fi}
\draw[->-=.5,Green,densely dotted] (x-1)--(x-4);
\draw[->-=.5,Green,densely dotted] (x-4)--(x-7);
\draw[->-=.5,Green,densely dotted] (x-7)--(x-10);
\draw[->-=.5,Green,densely dotted] (x-10)--(x-1);
\draw[->-=.5,Green,densely dotted] (x-2)--(x-5);
\draw[->-=.5,Green,densely dotted] (x-5)--(x-8);
\draw[->-=.5,Green,densely dotted] (x-8)--(x-11);
\draw[->-=.5,Green,densely dotted] (x-11)--(x-2);
\draw[->-=.5,Green,densely dotted] (x-3)--(x-6);
\draw[->-=.5,Green,densely dotted] (x-6)--(x-9);
\draw[->-=.5,Green,densely dotted] (x-9)--(x-12);
\draw[->-=.5,Green,densely dotted] (x-12)--(x-3);
\end{tikzpicture}
\qquad
\begin{tikzpicture}[scale=0.65, every node/.style={circle, draw, fill=black, minimum size=4pt, inner sep=0pt}]
\node at (1,1)(1) {}; 
\node (2) at (3.5,0) {};
\node(3) at (1,-1) {}; 
\node (4) at (2.5,1) {};
\node (5) at (2.5,-1) {}; 
\node  at (0,0)(6)  {};
\node (7) at (3.5,2) {};
\node (8) at (3.5,-2) {};
\node (9) at (-1.5,0) {};
\node (10)  at (0,2) {};
\node (11) at (5,0) {};
\node (12) at (0,-2) {};
\draw[->-=.5,blue,line width=0.3mm] (1) -- (4); 
\draw[->-=.5,blue,line width=0.3mm] (4) -- (2); 
\draw[->-=.5,blue,line width=0.3mm] (2) -- (5); 
\draw[->-=.5,blue,line width=0.3mm] (5) -- (3); 
\draw[->-=.5,blue,line width=0.3mm] (3) -- (6); 
\draw[->-=.5,blue,line width=0.3mm] (6) -- (1); 
\draw[->-=.5,blue,line width=0.3mm] (10) -- (7); 
\draw[->-=.5,blue,line width=0.3mm] (7) -- (11); 
\draw[->-=.5,blue,line width=0.3mm] (11) -- (8);
\draw[->-=.5,blue,line width=0.3mm] (8) -- (12); 
\draw[->-=.5,blue,line width=0.3mm] (12) -- (9); 
\draw[->-=.5,blue,line width=0.3mm] (9) -- (10); 
\draw[->-=.5,Green,densely dotted] (10) -- (4); 
\draw[->-=.5,Green,densely dotted] (4) -- (11); 
\draw[->-=.5,Green,densely dotted] (11) -- (5);
\draw[->-=.5,Green,densely dotted] (5) -- (12);
\draw[->-=.5,Green,densely dotted] (12) -- (6);
\draw[->-=.5,Green,densely dotted] (6) -- (10);
\draw[->-=.5,Green,densely dotted] (9) -- (1);
\draw[->-=.5,Green,densely dotted] (1) -- (7);
\draw[->-=.5,Green,densely dotted] (7) -- (2);
\draw[->-=.5,Green,densely dotted] (2) -- (8);
\draw[->-=.5,Green,densely dotted] (8) -- (3);
\draw[->-=.5,Green,densely dotted] (3) -- (9);
\end{tikzpicture}     
\qquad
      \begin{tikzpicture}[scale=0.65, every node/.style={circle, draw, fill=black, minimum size=4pt, inner sep=0pt}]
\foreach \X[count=\Y] in {black,black,black,black,black}
{\node[draw,circle,black] (x-\Y) at ({72*\Y+18}:2){}; }
\foreach \X[count=\Y] in {0,...,4}
{\ifnum\X=0
\draw[->-=.5,blue,line width=0.3mm] (x-5) --(x-\Y) ;
\else
\draw[->-=.5,blue,line width=0.3mm] (x-\X) --(x-\Y) ;
\fi}
\draw[->-=.5,Green,densely dotted] (x-1)--(x-4);
\draw[->-=.5,Green,densely dotted] (x-4)--(x-2);
\draw[->-=.5,Green,densely dotted] (x-2)--(x-5);
\draw[->-=.5,Green,densely dotted] (x-5)--(x-3);
\draw[->-=.5,Green,densely dotted] (x-3)--(x-1); 
     \end{tikzpicture}
 \end{center}
Here, the blue edges are represented by thickened lines, whereas the green edges are represented by dotted lines.

An abstract tectonic crater can be regarded as the "crater" of a tectonic volcano (which was first introduced in \cite[Definition~6.20]{LM2}):

\begin{defn}\label{def:crater}
  A \textbf{tectonic $l$-volcano} $G$ is a directed graph whose set of vertices may be decomposed as a disjoint union $\VV(G)=\bigcup_{i=0}^\infty \VV_i$ in such a way that
  \begin{itemize}
      \item The out-degree of every vertex is $l+1$;
      \item  The subgraph generated by $\VV_0$ {(that is, the maximal subgraph containing $\VV_0$)} is an abstract tectonic crater; we shall refer to $\VV_0$ as the crater of $G$;
      \item If there is an edge between $v_i\in\VV_i$ and $v_j\in\VV_j$ then $i-j\in \{\pm 1\}$ or $i=j=0$;
      \item Every $v\in \VV_i$ with $i\ge 1$ has exactly one edge starting at $v$ and ending at a vertex in $\VV_{i-1}$ and $l$ edges starting at $v$ and ending at a vertex in $\VV_{i+1}$.
  \end{itemize}
  An \textbf{undirected tectonic $l$-volcano} is the image of a tectonic volcano under the forgetful map.
\end{defn}
Note that only edges between two vertices of $\VV_0$ (those in the crater) are assigned a color. Below is an example of a tectonic $3$-volcano whose crater is of parameter $(5,1,1,2)$, where we only show the vertices in $\bigcup_{i=0}^2\VV_i$.
\begin{center}
\begin{tikzpicture}[scale=0.5, every node/.style={circle, draw, fill=black, minimum size=4pt, inner sep=0pt}]
\foreach \X[count=\Y] in {black,black,black,black,black} {
    \node[draw,circle,black] (x-\Y) at ({72*\Y+18}:2) {};
}
\foreach \X[count=\Y] in {0,...,4}
{\ifnum\X=0
\draw[->-=.5,blue,line width=0.3mm] (x-5) --(x-\Y) ;
\else
\draw[->-=.5,blue,line width=0.3mm] (x-\X) --(x-\Y) ;
\fi}

\draw[->-=.5,Green,densely dotted] (x-1) -- (x-4);
\draw[->-=.5,Green,densely dotted] (x-4) -- (x-2);
\draw[->-=.5,Green,densely dotted] (x-2) -- (x-5);
\draw[->-=.5,Green,densely dotted] (x-5) -- (x-3);
\draw[->-=.5,Green,densely dotted] (x-3) -- (x-1);
\foreach \i in {1,...,5} {
    \node[draw,circle,black] (y\i-1) at ($ (x-\i) + ({72*\i+35}:3) $) {};
    \node[draw,circle,black] (y\i-2) at ($ (x-\i) + ({72*\i-35}:3) $) {};
    \draw[->-=.5] (x-\i) to[bend left=-20] (y\i-1);
    \draw[->-=.5] (x-\i) to[bend left=-20] (y\i-2);
    \draw[->-=.5] (y\i-1) to[bend left=-20] (x-\i);
    \draw[->-=.5] (y\i-2) to[bend left=-20] (x-\i);
}

\foreach \i in {1,...,5} {
 
    \node[draw,circle,black] (z\i-1a) at ($ (y\i-1) + ({72*\i+45}:2) $) {};
    \node[draw,circle,black] (z\i-1b) at ($ (y\i-1) + ({72*\i}:2) $) {};
    \node[draw,circle,black] (z\i-1c) at ($ (y\i-1) + ({72*\i-45}:2) $) {};
    \draw[->-=.5] (y\i-1) to[bend left=-20] (z\i-1a);
    \draw[->-=.5] (y\i-1) to[bend left=-20] (z\i-1b);
    \draw[->-=.5] (y\i-1) to[bend left=-20] (z\i-1c);
    \draw[->-=.5] (z\i-1a) to[bend left=-20] (y\i-1);
    \draw[->-=.5] (z\i-1b) to[bend left=-20] (y\i-1);
    \draw[->-=.5] (z\i-1c) to[bend left=-20] (y\i-1);
    
    \node[draw,circle,black] (z\i-2a) at ($ (y\i-2) + ({72*\i+45}:2) $) {};
    \node[draw,circle,black] (z\i-2b) at ($ (y\i-2) + ({72*\i}:2) $) {};
    \node[draw,circle,black] (z\i-2c) at ($ (y\i-2) + ({72*\i-45}:2) $) {};
    \draw[->-=.5] (y\i-2) to[bend left=-20] (z\i-2a);
    \draw[->-=.5] (y\i-2) to[bend left=-20] (z\i-2b);
    \draw[->-=.5] (y\i-2) to[bend left=-20] (z\i-2c);
    \draw[->-=.5] (z\i-2a) to[bend left=-20] (y\i-2);
    \draw[->-=.5] (z\i-2b) to[bend left=-20] (y\i-2);
    \draw[->-=.5] (z\i-2c) to[bend left=-20] (y\i-2);
}
\end{tikzpicture}
\end{center}

For simplicity, we have chosen an example where for each edge between $\VV_i$ and $\VV_{i+1}$, there is an edge going in the opposite direction, but this is not part of the definition. For example, we can also have:

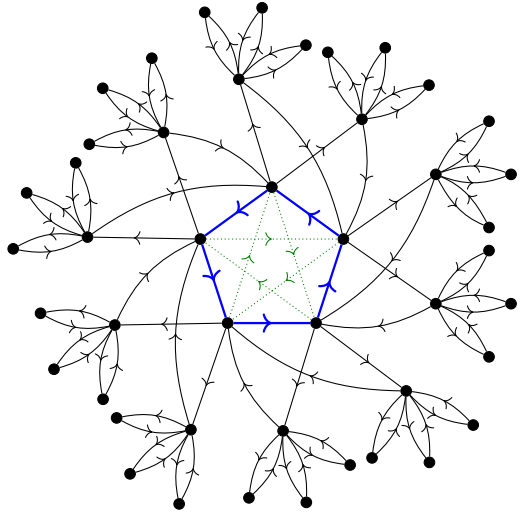
\begin{center}
\begin{tikzpicture}[scale=0.5, every node/.style={circle, draw, fill=black, minimum size=4pt, inner sep=0pt}]
\foreach \X[count=\Y] in {black,black,black,black,black} {
    \node[draw,circle,black] (x-\Y) at ({72*\Y+18}:2) {};
}
\foreach \X[count=\Y] in {0,...,4}
{\ifnum\X=0
\draw[->-=.5,blue,line width=0.3mm] (x-5) --(x-\Y) ;
\else
\draw[->-=.5,blue,line width=0.3mm] (x-\X) --(x-\Y) ;
\fi}

\draw[->-=.5,Green,densely dotted] (x-1) -- (x-4);
\draw[->-=.5,Green,densely dotted] (x-4) -- (x-2);
\draw[->-=.5,Green,densely dotted] (x-2) -- (x-5);
\draw[->-=.5,Green,densely dotted] (x-5) -- (x-3);
\draw[->-=.5,Green,densely dotted] (x-3) -- (x-1);
\foreach \i in {1,...,5} {
    \node[draw,circle,black] (y\i-1) at ($ (x-\i) + ({72*\i+35}:3) $) {};
    \node[draw,circle,black] (y\i-2) at ($ (x-\i) + ({72*\i-35}:3) $) {};
    \draw[->-=.6] (x-\i) to (y\i-1);
    \draw[->-=.6] (x-\i) to (y\i-2);
   
}

\draw[->-=.5] (y1-1) to[bend left=20] (x-5);
\draw[->-=.5] (y1-2) to[bend left=20] (x-5);
\draw[->-=.5] (y2-1) to[bend left=20] (x-1);
\draw[->-=.5] (y2-2) to[bend left=20] (x-1);
\draw[->-=.5] (y3-1) to[bend left=20] (x-2);
\draw[->-=.5] (y3-2) to[bend left=20] (x-2);
\draw[->-=.5] (y4-1) to[bend left=20] (x-3);
\draw[->-=.5] (y4-2) to[bend left=20] (x-3);
\draw[->-=.5] (y5-1) to[bend left=20] (x-4);
\draw[->-=.5] (y5-2) to[bend left=20] (x-4);

\foreach \i in {1,...,5} {
 
    \node[draw,circle,black] (z\i-1a) at ($ (y\i-1) + ({72*\i+45}:2) $) {};
    \node[draw,circle,black] (z\i-1b) at ($ (y\i-1) + ({72*\i}:2) $) {};
    \node[draw,circle,black] (z\i-1c) at ($ (y\i-1) + ({72*\i-45}:2) $) {};
    \draw[->-=.5] (y\i-1) to[bend left=-20] (z\i-1a);
    \draw[->-=.5] (y\i-1) to[bend left=-20] (z\i-1b);
    \draw[->-=.5] (y\i-1) to[bend left=-20] (z\i-1c);
    \draw[->-=.5] (z\i-1a) to[bend left=-20] (y\i-1);
    \draw[->-=.5] (z\i-1b) to[bend left=-20] (y\i-1);
    \draw[->-=.5] (z\i-1c) to[bend left=-20] (y\i-1);
    
    \node[draw,circle,black] (z\i-2a) at ($ (y\i-2) + ({72*\i+45}:2) $) {};
    \node[draw,circle,black] (z\i-2b) at ($ (y\i-2) + ({72*\i}:2) $) {};
    \node[draw,circle,black] (z\i-2c) at ($ (y\i-2) + ({72*\i-45}:2) $) {};
    \draw[->-=.5] (y\i-2) to[bend left=-20] (z\i-2a);
    \draw[->-=.5] (y\i-2) to[bend left=-20] (z\i-2b);
    \draw[->-=.5] (y\i-2) to[bend left=-20] (z\i-2c);
    \draw[->-=.5] (z\i-2a) to[bend left=-20] (y\i-2);
    \draw[->-=.5] (z\i-2b) to[bend left=-20] (y\i-2);
    \draw[->-=.5] (z\i-2c) to[bend left=-20] (y\i-2);
}
\end{tikzpicture}
\end{center}

Finally, we define $l$-volcano graphs.

    \begin{defn}\label{def:volcano}
\item[i)] An \textbf{abstract crater} is either a connected and directed cycle graph or a totally disconnected finite graph.
\item[ii)]      An $l$-\textbf{volcano}  is a directed graph $G$ whose vertices may be decomposed as a disjoint union $\VV(G)=\bigcup_{i=0}^\infty \VV_i$ in such a way that
\begin{itemize}
    \item The out-degree of every vertex is $l+1$;
    \item The subgraph generated by $\VV_0$ is  an abstract crater;
    \item  If there is an edge between $v_i$ and $v_j$ then $i-j\in \{\pm 1\}$ or $i=j=0$;
    \item Each $v\in \VV_i$ with $i\ge 1$ has exactly one edge starting at $v$ and ending at a vertex in $\VV_{i-1}$ and $l$ edges starting at $v$ and ending at vertices in $\VV_{i+1}$.
\end{itemize} 
 The image of an $l$-volcano under the forgetful map is called an \textbf{undirected $l$-volcano}.
\end{defn}
Note that an $l$-volcano defined here is not exactly the same as the volcano graphs considered in \cite{kohel,pazuki}. Since the crater is a directed cycle graph, the requirement that each vertex has out-degree $l+1$ means that there are $l$ edges going from a vertex in $\VV_0$ to the vertices in $\VV_1$. In the works \cite{kohel,pazuki}, a directed volcano graph would consist of a cycle with edges going into both directions, and there would only be $l-1$ edges going from a vertex of the crater to "depth-one" vertices. For example, when the crater contains 4 vertices and $l=3$, the subgraph generated by the vertices of depth 0, 1 and 2 can be represented as:
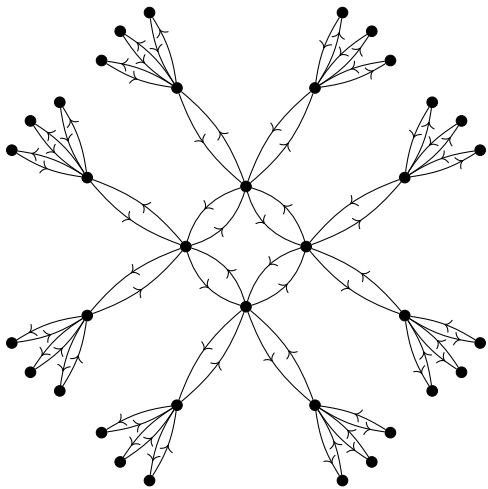
\begin{center}
\begin{tikzpicture}[scale=0.8, every node/.style={circle, draw, fill=black, minimum size=4pt, inner sep=0pt}]
\def\n{4}
\def\R{1}
\def\L{2}
\foreach \i in {1,...,\n} {
    \node (C\i) at ({360/\n * (\i - 1)}:\R) {};
    \draw[->-=.5] (C\i) to[bend left=-30] ({360/\n * (\i)}:\R);
    \draw[->-=.5] ({360/\n * (\i)}:\R) to[bend left=-30] (C\i);
}
\path (C\n) to[bend left=-30] (C1);
\foreach \i in {1,...,\n} {
        \node (T\i-1a) at ($ (C\i) + ({360/\n * (\i - 1) + 35}:\L) $) {};
    \node (T\i-1b) at ($ (C\i) + ({360/\n * (\i - 1) - 35}:\L) $) {};
      
    \draw[->-=.5] (C\i)to[bend left=-15] (T\i-1a);
    \draw[->-=.5] (T\i-1a)to[bend left=-15] (C\i);
      \draw[->-=.5] (C\i)to[bend left=-15] (T\i-1b);
    \draw[->-=.5] (T\i-1b)to[bend left=-15] (C\i);
      
       \node (T\i-2a1) at ($ (T\i-1a) + ({360/\n * (\i - 1) + 70}:\L/1.5) $) {};
    \node (T\i-2a2) at ($ (T\i-1a) + ({360/\n * (\i - 1)+45}:\L/1.5) $) {};
    \node (T\i-2a3) at ($ (T\i-1a) + ({360/\n * (\i - 1) +20}:\L/1.5) $) {};
    \draw[->-=.8] (T\i-1a) to[bend left=-12] (T\i-2a1);
    \draw[->-=.5] (T\i-2a1) to[bend left=-12] (T\i-1a);
    \draw[->-=.8] (T\i-1a) to[bend left=-12] (T\i-2a2);
    \draw[->-=.5] (T\i-2a2) to[bend left=-12] (T\i-1a);
    \draw[->-=.8] (T\i-1a) to[bend left=-12] (T\i-2a3);
    \draw[->-=.5] (T\i-2a3) to[bend left=-12] (T\i-1a);
          \node (T\i-2b1) at ($ (T\i-1b) + ({360/\n * (\i - 1) - 20}:\L/1.5) $) {};
    \node (T\i-2b2) at ($ (T\i-1b) + ({360/\n * (\i - 1)-45}:\L/1.5) $) {};
    \node (T\i-2b3) at ($ (T\i-1b) + ({360/\n * (\i - 1) - 70}:\L/1.5) $) {};
    \draw[->-=.8] (T\i-1b) to[bend left=-12] (T\i-2b1);
    \draw[->-=.5] (T\i-2b1) to[bend left=-12] (T\i-1b);
    \draw[->-=.8] (T\i-1b) to[bend left=-12] (T\i-2b2);
    \draw[->-=.5] (T\i-2b2) to[bend left=-12] (T\i-1b);
    \draw[->-=.8] (T\i-1b) to[bend left=-12] (T\i-2b3);
    \draw[->-=.5] (T\i-2b3) to[bend left=-12] (T\i-1b);
   }
\end{tikzpicture}
\end{center}

For comparison, below is an illustration of two $3$-volcanoes, where the crater is a directed cycle graph with $4$ vertices. Again, we only show the vertices in $\bigcup_{i=0}^2\VV_i$.
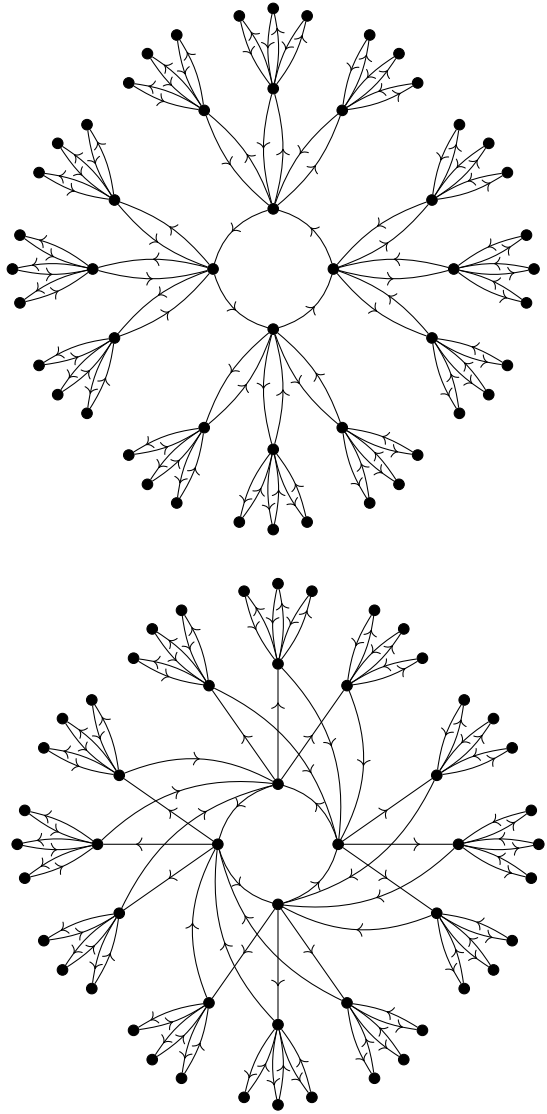
\begin{center}
\begin{tikzpicture}[scale=0.8, every node/.style={circle, draw, fill=black, minimum size=4pt, inner sep=0pt}]
\def\n{4}
\def\R{1}
\def\L{2}
\foreach \i in {1,...,\n} {
    \node (C\i) at ({360/\n * (\i - 1)}:\R) {};
    \draw[->-=.5] (C\i) to[bend left=-30] ({360/\n * (\i)}:\R);
}
\path (C\n) to[bend left=-30] (C1);
\foreach \i in {1,...,\n} {
        \node (T\i-1a) at ($ (C\i) + ({360/\n * (\i - 1) + 35}:\L) $) {};
    \node (T\i-1b) at ($ (C\i) + ({360/\n * (\i - 1) - 35}:\L) $) {};
      \node (T\i-1c) at ($ (C\i) + ({360/\n * (\i - 1) }:\L) $) {};
    \draw[->-=.5] (C\i)to[bend left=-15] (T\i-1a);
    \draw[->-=.5] (T\i-1a)to[bend left=-15] (C\i);
      \draw[->-=.5] (C\i)to[bend left=-15] (T\i-1b);
    \draw[->-=.5] (T\i-1b)to[bend left=-15] (C\i);
      \draw[->-=.5] (C\i)to[bend left=-15] (T\i-1c);
    \draw[->-=.5] (T\i-1c)to[bend left=-15] (C\i);
       \node (T\i-2a1) at ($ (T\i-1a) + ({360/\n * (\i - 1) + 70}:\L/1.5) $) {};
    \node (T\i-2a2) at ($ (T\i-1a) + ({360/\n * (\i - 1)+45}:\L/1.5) $) {};
    \node (T\i-2a3) at ($ (T\i-1a) + ({360/\n * (\i - 1) +20}:\L/1.5) $) {};
    \draw[->-=.8] (T\i-1a) to[bend left=-12] (T\i-2a1);
    \draw[->-=.5] (T\i-2a1) to[bend left=-12] (T\i-1a);
    \draw[->-=.8] (T\i-1a) to[bend left=-12] (T\i-2a2);
    \draw[->-=.5] (T\i-2a2) to[bend left=-12] (T\i-1a);
    \draw[->-=.8] (T\i-1a) to[bend left=-12] (T\i-2a3);
    \draw[->-=.5] (T\i-2a3) to[bend left=-12] (T\i-1a);
          \node (T\i-2b1) at ($ (T\i-1b) + ({360/\n * (\i - 1) - 20}:\L/1.5) $) {};
    \node (T\i-2b2) at ($ (T\i-1b) + ({360/\n * (\i - 1)-45}:\L/1.5) $) {};
    \node (T\i-2b3) at ($ (T\i-1b) + ({360/\n * (\i - 1) - 70}:\L/1.5) $) {};
    \draw[->-=.8] (T\i-1b) to[bend left=-12] (T\i-2b1);
    \draw[->-=.5] (T\i-2b1) to[bend left=-12] (T\i-1b);
    \draw[->-=.8] (T\i-1b) to[bend left=-12] (T\i-2b2);
    \draw[->-=.5] (T\i-2b2) to[bend left=-12] (T\i-1b);
    \draw[->-=.8] (T\i-1b) to[bend left=-12] (T\i-2b3);
    \draw[->-=.5] (T\i-2b3) to[bend left=-12] (T\i-1b);
       \node (T\i-2c1) at ($ (T\i-1c) + ({360/\n * (\i - 1) + 25}:\L/1.5) $) {};
    \node (T\i-2c2) at ($ (T\i-1c) + ({360/\n * (\i - 1)}:\L/1.5) $) {};
    \node (T\i-2c3) at ($ (T\i-1c) + ({360/\n * (\i - 1) - 25}:\L/1.5) $) {};
   \draw[->-=.8] (T\i-1c) to[bend left=-12] (T\i-2c1);
    \draw[->-=.5] (T\i-2c1) to[bend left=-12] (T\i-1c);
    \draw[->-=.8] (T\i-1c) to[bend left=-12] (T\i-2c2);
    \draw[->-=.5] (T\i-2c2) to[bend left=-12] (T\i-1c);
    \draw[->-=.8] (T\i-1c) to[bend left=-12] (T\i-2c3);
    \draw[->-=.5] (T\i-2c3) to[bend left=-12] (T\i-1c);
}
\end{tikzpicture}
\vspace{0.3cm}

\begin{tikzpicture}[scale=0.8, every node/.style={circle, draw, fill=black, minimum size=4pt, inner sep=0pt}]
\def\n{4}
\def\R{1}
\def\L{2}
\foreach \i in {1,...,\n} {
    \node (C\i) at ({360/\n * (\i - 1)}:\R) {};
    \draw[->-=.5] (C\i) to[bend left=-30] ({360/\n * (\i)}:\R);
    \draw[->-=.5](C\i)to (T\i-1a) ;
     \draw[->-=.5] (C\i)to(T\i-1b) ;
    \draw[->-=.7] (C\i)to(T\i-1c) ;

}
\path (C\n) to[bend left=-30] (C1);
\foreach \i in {1,...,\n} {
        \node (T\i-1a) at ($ (C\i) + ({360/\n * (\i - 1) + 35}:\L) $) {};
    \node (T\i-1b) at ($ (C\i) + ({360/\n * (\i - 1) - 35}:\L) $) {};
      \node (T\i-1c) at ($ (C\i) + ({360/\n * (\i - 1) }:\L) $) {};
       \node (T\i-2a1) at ($ (T\i-1a) + ({360/\n * (\i - 1) + 70}:\L/1.5) $) {};
    \node (T\i-2a2) at ($ (T\i-1a) + ({360/\n * (\i - 1)+45}:\L/1.5) $) {};
    \node (T\i-2a3) at ($ (T\i-1a) + ({360/\n * (\i - 1) +20}:\L/1.5) $) {};
    \draw[->-=.8] (T\i-1a) to[bend left=-12] (T\i-2a1);
    \draw[->-=.5] (T\i-2a1) to[bend left=-12] (T\i-1a);
    \draw[->-=.8] (T\i-1a) to[bend left=-12] (T\i-2a2);
    \draw[->-=.5] (T\i-2a2) to[bend left=-12] (T\i-1a);
    \draw[->-=.8] (T\i-1a) to[bend left=-12] (T\i-2a3);
    \draw[->-=.5] (T\i-2a3) to[bend left=-12] (T\i-1a);
          \node (T\i-2b1) at ($ (T\i-1b) + ({360/\n * (\i - 1) - 20}:\L/1.5) $) {};
    \node (T\i-2b2) at ($ (T\i-1b) + ({360/\n * (\i - 1)-45}:\L/1.5) $) {};
    \node (T\i-2b3) at ($ (T\i-1b) + ({360/\n * (\i - 1) - 70}:\L/1.5) $) {};
    \draw[->-=.8] (T\i-1b) to[bend left=-12] (T\i-2b1);
    \draw[->-=.5] (T\i-2b1) to[bend left=-12] (T\i-1b);
    \draw[->-=.8] (T\i-1b) to[bend left=-12] (T\i-2b2);
    \draw[->-=.5] (T\i-2b2) to[bend left=-12] (T\i-1b);
    \draw[->-=.8] (T\i-1b) to[bend left=-12] (T\i-2b3);
    \draw[->-=.5] (T\i-2b3) to[bend left=-12] (T\i-1b);
       \node (T\i-2c1) at ($ (T\i-1c) + ({360/\n * (\i - 1) + 25}:\L/1.5) $) {};
    \node (T\i-2c2) at ($ (T\i-1c) + ({360/\n * (\i - 1)}:\L/1.5) $) {};
    \node (T\i-2c3) at ($ (T\i-1c) + ({360/\n * (\i - 1) - 25}:\L/1.5) $) {};
   \draw[->-=.8] (T\i-1c) to[bend left=-12] (T\i-2c1);
    \draw[->-=.5] (T\i-2c1) to[bend left=-12] (T\i-1c);
    \draw[->-=.8] (T\i-1c) to[bend left=-12] (T\i-2c2);
    \draw[->-=.5] (T\i-2c2) to[bend left=-12] (T\i-1c);
    \draw[->-=.8] (T\i-1c) to[bend left=-12] (T\i-2c3);
    \draw[->-=.5] (T\i-2c3) to[bend left=-12] (T\i-1c);
}

\draw[->-=.65] (T2-1a)to[bend left=25](C1) ;   
\draw[->-=.5] (T2-1b)to[bend left=25](C1) ;
\draw[->-=.5] (T2-1c)to[bend left=25](C1) ;
\draw[->-=.65] (T3-1a)to[bend left=25](C2) ;   
\draw[->-=.5] (T3-1b)to[bend left=25](C2) ;
\draw[->-=.5] (T3-1c)to[bend left=25](C2) ;
\draw[->-=.65] (T4-1a)to[bend left=25](C3) ;   
\draw[->-=.5] (T4-1b)to[bend left=25](C3) ;
\draw[->-=.5] (T4-1c)to[bend left=25](C3) ;
\draw[->-=.65] (T1-1a)to[bend left=25](C4) ;   
\draw[->-=.5] (T1-1b)to[bend left=25](C4) ;
\draw[->-=.5] (T1-1c)to[bend left=25](C4) ;

\end{tikzpicture}

\end{center}

We note that the main difference between an $l$-volcano and a tectonic $l$-volcano is the structure of the subgraph generated by $\VV_0$. 

\subsection{Preliminary lemmas - I}
In this section, we prove several purely graph-theoretical results that will be used in the proofs of Theorem~\ref{thmB}.

\begin{defn}\label{def:color-function}
    Let $X$ be an abstract tectonic crater. Let $\pi:Y\rightarrow X$ be a graph covering. We assign each edge in $\EE(Y)$ a color, blue or green, according to the color of $\pi(e)$. As in $\VV(X)$, for each $v\in \VV(Y)$, there is exactly one blue/green edge with $v$ as the source/target. We define $\fb_Y(v)$ (resp. $\fg_Y(v)$) to be the target of the unique blue (resp. green) edge whose source is $v$.
\end{defn}

\begin{lemma}
\label{lem:galois-tectonic}
  Let $X$ be an abstract tectonic crater. Let $Y/X$ be a Galois covering of finite connected graphs. 
  If $\fb_Y\circ \fg_Y=\fg_Y\circ \fb_Y$, then $Y$ is an abstract tectonic crater. 
\end{lemma}
\begin{proof}
Let the parameters of $X$ as an abstract tectonic crater be $(\r,\s,\t,\cc)$. 
Let $X_{\textup{b}}$ (resp. $X_\textup{g}$) be the subgraph obtained by deleting all the green (resp. blue) edges in $X$. Then $X_{\textup{b}}$ (resp. $X_{\textup{g}}$) is the disjoint union of cycle graphs of length $\r\s$ (resp. $\r\t$). We define $Y_{\textup{b}}$ and $Y_{\textup{g}}$ analogously. Then $Y_{\tb}/X_\tb$ and $Y_\tg/X_\tg$ are graph coverings. 

We recall from {Theorem~\ref{thm:Gonet} that $Y$ is isomorphic to} the derived graph of a $\Gal(Y/X)$-valued voltage assignment, which we denote by $\alpha$. Let $\alpha_\tb$ (resp. $\alpha_\tg$) be the restriction of $\alpha$ to the blue (resp. green) edges of $X$. Then $Y_\tb$ (resp. $Y_\tg$) is the derived graph of $\alpha_\tb$ (resp. $\alpha_\tg$).

We recall from \cite[Proposition~2.4i)]{LM2} that $\Gal(Y/X)$ acts transitively on the connected components of $Y_\tb$. Therefore,  {$Y_\tb$} is the disjoint union of cycle graphs of the same length, which we denote by $s_1$. In particular, the in-degree and out-degree of each vertex of $Y_\tb$ are equal to 1. {Furthermore, each vertex in $Y_\tb$ lies on a unique blue cycle graph}.  The same can be said about $Y_{\textup{g}}$. Therefore, $Y$ satisfies properties b), c) and d) of Definition~\ref{def:crater-tec}.

Fix a vertex $v\in\VV(Y)$. Let $s$ be the minimal integer such that the following holds: there exist a path in $Y_\tb$ of length $s$ and a path in $Y_\tg$, both of which start at $v$ and terminate at the same vertex.
      As $\fb_Y\circ \fg_Y=\fg_Y\circ \fb_Y$, the integer $s$ is independent of the choice of $v$ within a fixed cycle of $Y_\fb$. Let $\sigma\in \Gal(Y/X)$. If we determine $s$ with respect to $\sigma(v)$ instead of $v$, we obtain the same value for $s$. As each cycle in $Y_\fb$ contains at least one Galois conjugate of the fixed vertex $v$, the integer $s$ does not depend on the choice of $v$. Furthermore, the minimality of $s$ implies that $s$ is a divisor of $s_1$.
      
      Let $r=s_1/s$. A cycle in $Y_\tb$ is decomposed into $r$ disjoint paths, each of which corresponds to a path of fixed length in $Y_\tg$. Therefore, the length of a cycle in $Y_\tg$ is given by $rt$ for some positive integer $t$. In particular, the number of connected components of $Y_{\textup{b}}$ is given by 
    \[\frac{tr}{r}=t.\]
    Thus, $Y$ contains in total $rst$ vertices, which verifies property a). 
    
    It remains to verify property e). Let $C_\tb$ (resp. $C_\tg$) be the unique cycle of length $rs$ in $Y_\tb$ (resp. $rt$ in $Y_\tg$) passing through $v\in \VV(Y)$. Then $C_\tb$ and $C_\tg$ intersect at $r$ vertices $\{v_1,\dots v_r\}$. Let $w=v_i$ for some $i$ such that there is a path of length $s$ in $Y_\tb$  going from $v$ to $w$. Let $c'$ be the length of the shortest path in $Y_\tg$ going from $v$ to $w$. Then the minimality of $c'$ implies that there is a cycle of length $rc'$ in $Y_\tg$ passing through all the vertices in $\{v_1,\dots,v_r\}$. It follows that $rt\mid rc'$. In particular, $t\mid c'$. Let $c=c'/t$. Then $C_\tb$ and $C_\tg$ meet at every $s$ (resp. $ct$) steps along  $C_\tb$ (resp. $C_\tg$). Hence, we conclude that $Y$ is an abstract tectonic crater of parameters $(r,s,t,c)$.
\end{proof}
\begin{example}
    Let $X$ be the tectonic crater with parameters $(5,1,1,2)$ as in \cite[Example~6.19]{LM2}. Let $\alpha$ be the voltage assignment on $X$ that assigns to each edge the non-trivial element in $\ZZ/2\ZZ$ and let $Y=X(\ZZ/2\ZZ,\alpha)$. Then $Y$ is connected and $Y/X$ is Galois with Galois group $\ZZ/2\ZZ=\langle \sigma \rangle$. Let $v$ be a vertex of $X$ and $\tau\in \ZZ/2\ZZ$, then \[\fb_Y\circ \fg_Y((v,\tau))=(\fb_X\circ \fg_X(v),\tau\sigma^2)=(\fg_X\circ \fb_X (v),\tau\sigma^2)=\fg_Y\circ \fb_Y((v,\tau)).\]
    We can check by direct computation that $Y$ is a tectonic crater of parameters $(10,1,1,7)$:
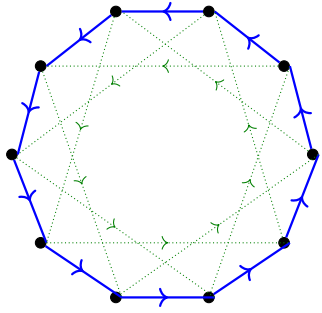
\begin{center}
    \begin{tikzpicture}[scale=1.0, every node/.style={circle, draw, fill=black, minimum size=4pt, inner sep=0pt}]

\def\n{10} 
\def\R{2}  
\foreach \i in {1,...,\n} {
    \node (x-\i) at ({360/\n * (\i - 1)}:\R) {};
}
\foreach \i [evaluate=\i as \next using {mod(\i,\n)+1}] in {1,...,\n} {
    \draw[->-=.5,blue,line width=0.3mm] (x-\i) -- (x-\next);
}
\foreach \i [evaluate=\i as \third using {mod(\i+2,\n)+1}] in {1,...,\n} {
    \draw[->-=.5,Green,densely dotted] (x-\i) -- (x-\third);
}
\end{tikzpicture}
\end{center}

    Let $\beta$ be a voltage assignment on $X$ that sends a fixed edge of $X$ a generator of $\ZZ/3\ZZ$, and sends all the other edges to the identity element of $\ZZ/3\ZZ$. Let $Y'=X(\ZZ/3\ZZ,\beta)$. Then $Y'$ is as follows:
     
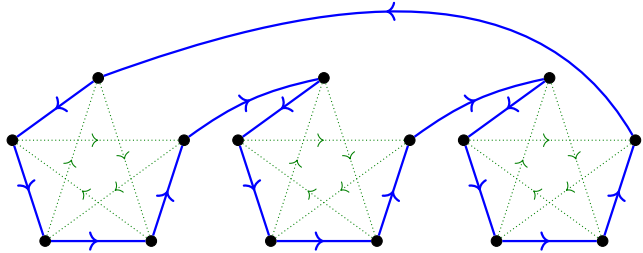
\begin{center}
     \begin{tikzpicture}[scale=0.6, every node/.style={circle, draw, fill=black, minimum size=4pt, inner sep=0pt}]
\foreach \X[count=\Y] in {black,black,black,black,black}
{\node[draw,circle,black] (x-\Y) at ({72*\Y+18}:2){}; }
\foreach \X[count=\Y] in {0,...,4}
{\ifnum\X=0
\else
\draw[->-=.5,blue,line width=0.3mm] (x-\X) --(x-\Y) ;
\fi}
\draw[->-=.5,Green,densely dotted] (x-1)--(x-4);
\draw[->-=.5,Green,densely dotted] (x-4)--(x-2);
\draw[->-=.5,Green,densely dotted] (x-2)--(x-5);
\draw[->-=.5,Green,densely dotted] (x-5)--(x-3);
\draw[->-=.5,Green,densely dotted] (x-3)--(x-1); 

\begin{scope}[shift={(5,0)}]
\foreach \X[count=\Y] in {black,black,black,black,black}
{\node[draw,circle,black] (x2-\Y) at ({72*\Y+18}:2){}; }
\foreach \X[count=\Y] in {0,...,4}
{\ifnum\X=0
\else
\draw[->-=.5,blue,line width=0.3mm] (x2-\X) --(x2-\Y) ;
\fi}
\draw[->-=.5,Green,densely dotted] (x2-1)--(x2-4);
\draw[->-=.5,Green,densely dotted] (x2-4)--(x2-2);
\draw[->-=.5,Green,densely dotted] (x2-2)--(x2-5);
\draw[->-=.5,Green,densely dotted] (x2-5)--(x2-3);
\draw[->-=.5,Green,densely dotted] (x2-3)--(x2-1);

\end{scope}

\begin{scope}[shift={(10,0)}]
\foreach \X[count=\Y] in {black,black,black,black,black}
{\node[draw,circle,black] (x3-\Y) at ({72*\Y+18}:2){}; }
\foreach \X[count=\Y] in {0,...,4}
{\ifnum\X=0
\else
\draw[->-=.5,blue,line width=0.3mm] (x3-\X) --(x3-\Y) ;
\fi}
\draw[->-=.5,Green,densely dotted] (x3-1)--(x3-4);
\draw[->-=.5,Green,densely dotted] (x3-4)--(x3-2);
\draw[->-=.5,Green,densely dotted] (x3-2)--(x3-5);
\draw[->-=.5,Green,densely dotted] (x3-5)--(x3-3);
\draw[->-=.5,Green,densely dotted] (x3-3)--(x3-1);
\end{scope}

\draw[->-=.5,blue,line width=0.3mm] (x-5) to[bend left=12] (x2-1) ;
\draw[->-=.5,blue,line width=0.3mm] (x2-5) to[bend left=12](x3-1) ;
\draw[->-=.5,blue,line width=0.3mm] (x3-5) to[out=120,in=20](x-1) ;
     \end{tikzpicture}
     \end{center}

In this case, $Y'/X$ is again Galois and is visibly connected. However, one can verify that $\fg_{Y'}\circ \fb_{Y'}\neq \fb_{Y'}\circ \fg_{Y'}$ and that $Y'$ is not a tectonic crater, as property e) is not satisfied.
\end{example}
 
\begin{lemma}
\label{lem:galois-non-tectonic}
    Let $Y/X$ be a Galois covering of finite directed graphs. Assume that $X$ is a directed cycle graph and that $Y/X$ is not $2$-sheeted. Then $Y$ is a directed cycle graph. 
\end{lemma}
\begin{proof}
    Since $X$ is a directed cycle, the out-degree and in-degree of each vertex of $Y$ are equal to $1$. This implies that $Y$ is the disjoint union of cycle graphs. Furthermore, since $Y/X$ is Galois, each cycle has the same length.
    
   The following argument is also given in \cite[pages 65-66]{gonet-thesis} for undirected graphs. We recall it here for the convenience of the reader. Let $m$ be the length of a cycle in $Y$, let $s$ be the number of connected components in $Y$, and let $k$ be the number of vertices in $X$. In particular, $Y/X$ is a $(sm/k)$-sheeted covering and $m$ is divisible by $k$. 
    
    Assume that $(m/k)>1$. Each connected component $Z$ of $Y$ is a covering of $X$ and there are $m/k$ elements in $\textup{Deck}(Z/X)$. If $s>1$, then 
    \[\vert \textup{Deck}(Y/X)\vert \ge s!(m/k)^s>(sm/k),\]
    which contradicts that $Y/X$ is Galois. Therefore, $s=1$ and $Y$ is a directed cycle.

    Assume that $m/k=1$. In this case, the covering $Y/X$ is $s$-sheeted and there are $s!$ deck transformations. Thus, for $Y/X$ to be Galois, we must have $s=2$. This contradicts our assumption that $Y/X$ is not $2$-sheeted.
\end{proof}

\subsection{Preliminary lemmas - II}
Recall that $\mathcal{G}_l^p(n,N)$ is the isogeny graph of oriented supersingular elliptic curves defined in Definition~\ref{def:oriented-graph}. Let $\mathcal{Y}_0$ be a fixed connected component of $\mathcal{G}_l^p(0,N)$. For $n\ge0$, let $\mathcal{Y}_n$ be a connected component of $\cG_l^p(n,N)$ whose image under the natural map $(E,\iota,Q,P)\mapsto (E,\iota,Q)$ contains $\mathcal{Y}_0$. 

 Let $(E,\iota,Q,P)\in\VV(\mathcal{Y}_n)$ for some $\iota:K\hookrightarrow \End(E)\otimes\QQ$. Note that $K$ is independent of the choice of $E$ {by the properties iv)-v) in Definition~\ref{def:oriented-elliptic}}. We shall call $K$ the \textbf{\textit{CM field}} of $\mathcal{Y}_n$. It is clear from the definition that the CM field is independent of $n$.

We shall prove that $\mathcal{Y}_n$ is an undirected (tectonic) $l$-volcano. We begin with the case $n=0$.
\begin{lemma}\label{lem:l-volcano}
    The graph $\mathcal{Y}_0$ is an undirected (tectonic) $l$-volcano. 
\end{lemma}
\begin{proof}
If $N=1$, $\mathcal{Y}_0$ is indeed an $l$-volcano by \cite[Proposition~4.1]{onuki} and \cite[Theorem~1]{arpin-win}. 

Assume now that $N\ge 2$. {The result follows from \cite[proofs of Theorems~6.21, 6.22 and 6.23]{LM2}, which deal with the cases where $l$ is split, ramified and inert in $K$ separately. We outline the main ideas here for the convenience of the reader. Let $\mathfrak{Y}_0$ be the pre-image of $\mathcal{Y}_0$ in $\mathfrak{G}_l^p(0,N)$ under the forgetful map $\mathfrak{G}_l^p(0,N)\rightarrow\cG_l^p(n,N)$.}

Let $K$ be the CM field of $\cY_0$. Suppose that $\ell$ is split in $K$.
  Let $v=(E,\iota,Q,P)\in\VV(\fY_0)$. Let $\cO$ be a primitive order with respect to $(E,\iota)$.  Since $(E,\iota)$ does not admit ascending isogenies, $l$ does not divide the conductor of $\cO$. Thus, $l$ splits into two distinct ideals in $\cO$. Starting at $v$, and propagating along horizontal edges, we obtain a finite connected subgraph $\cZ_0$ of $\cY_0$. Each edge corresponds to one of the two distinct primes of $\cO$ lying above $l$. This allows us to assign to each edge a color - blue or green. By analyzing the action of special elements of $K$ on $E[l]$, we can show that the corresponding coloring functions (in the sense of Definition~\ref{def:color-function}) commute and check that $\cZ_0$ is an abstract tectonic crater. The other vertices of $\fY_0$ would have positive depths. Basic properties of $l$-isogenies allow us to verify that $\fY_0$ is indeed a tectonic $l$-volcano.

When $l$ is ramified in $K$, the construction of $\cZ_0$ above would result in a directed cycle - i.e. an abstract crater, rather than an abstract tectonic crater. In the inert case, there are no horizontal edges, which means that $\cZ_0$ consists of isolated vertices without any edges. We remark that it is assumed in \cite{LM2}  that $N>2$. In fact, this condition is not used in the construction of $\cZ_0$ outlined above. Furthermore, the concept of "double intertwinment" does not arise in the context of the present article. The reason is that the triples $(E,\iota, Q)$ and $(E,\iota, -Q)$ give rise to the same vertex here, whereas in \emph{loc. cit.}, they correspond to two distinct vertices. 
\end{proof}
\begin{remark}
\label{rem:level-preserving}
    The depth of a vertex of $\mathcal{Y}_0$ does not depend on $N$ and $n$. Thus, the covering $\pi_n\colon \mathcal{Y}_n\to \mathcal{Y}_0$ satisfies
    \[\textup{depth}(v)=\textup{depth}(\pi_n(v)),\quad\forall v\in\VV(\mathcal{Y}_n)
    \]
    In particular, the depth zero subgraph of $\mathcal{Y}_n$ is mapped to the crater of $\mathcal{Y}_0$ and each vertex of level at least $1$ in $\mathcal{Y}_n$ admits exactly one ascending edge and $l$ descending edges. 
\end{remark}

\begin{lemma}\label{lem:Gal} 
   The covering $\mathcal{Y}_n/\mathcal{Y}_0$ is Galois. The Galois group is isomorphic to a subgroup of $(\ZZ/p^n\ZZ)^\times$.
\end{lemma}
\begin{proof}
    By definition, $\mathcal{Y}_n/\mathcal{Y}_0$ is a $p^{n-1}(p-1)$-sheeted covering. Let $\mathcal{X}_0=\partial(\cY_0)${, where $\partial$ is given by Definition~\ref{def:partial}, }and $\mathcal{X}_n=\partial(\mathcal{Y}_n)$.
    
     Recall from the proof of Theorem~\ref{thmA} that $\mathcal{X}_n$ is a connected component of the derived graph of a voltage assignment $\alpha_n$ on $\mathcal{X}_0$. Let $\beta_n=\alpha_n\circ\partial$, which is a voltage assignment on $\mathcal{Y}_0$. We can identify $\mathcal{Y}_n$ with a connected component of $X(\beta_n,\Z/p^n\Z)$. Thus, it follows from Lemma \ref{connected component galois} that $\mathcal{Y}_n/\mathcal{Y}_0$ is Galois. 
\end{proof}

\begin{lemma}
\label{lem:galois}
Let $Z_0$ be the crater of $\mathcal{Y}_0$ and let $Z_n$ be the subgraph of depth zero vertices of $\mathcal{Y}_n$. Assume that $Z_0$ is connected, then $Z_n/Z_0$ is a Galois covering. The Galois group is isomorphic to a subgroup of $(\ZZ/p^n\ZZ)^\times$.
\end{lemma}
\begin{remark}
\label{remark-connected}
    Suppose that $Z_0$ consists of more than one vertex. We have seen in the proof of Lemma~\ref{lem:l-volcano} that $Z_0$ is connected only if $l$ is split or ramified in $K$.  Let $v$ and $w$ be two distinct vertices in $Z_0$, and let $\fP$ be a path in $\cY_0$ going from $v$ to $w$. We can see from the structure of a (tectonic) $l$-volcano that a maximal sequence of consecutive descending edges in $\fP$ is always followed by a sequence of ascending edges of the same length. The combination of these two paths is equivalent to composing an isogeny of $l$-power degree with its dual. In other words, it is equivalent to applying $[l^m]$ for some integer $m$. We may replace $[l]$ by $\widehat{\phi}\circ\phi$, where $\phi$ is a horizontal $\ell$-isogeny (which exists when $l$ is split or ramified in $K$). Therefore, we may replace $\fP$ with a path consisting of only edges in $Z_0$. 
\end{remark}
\begin{proof}
    By Remark \ref{rem:level-preserving}, $Z_n/Z_0$ is a graph covering, and by Lemma~\ref{lem:Gal}, the cover $\mathcal{Y}_n/\mathcal{Y}_0$ is Galois with an abelian Galois group. Therefore,  $Z_n=X(\alpha,G)$ for some voltage assignment $\alpha$ on $Z_0$ taking values in an abelian group $G$. 
    
    Let $v$ and $w$ be two vertices in $Z_n$. As $\mathcal{Y}_n$ is connected, there exists a path $\mathfrak{P}$ from $v$ to $w$. {The same argument as in Remark \ref{remark-connected} implies that} we may assume that $\fP$ consists only of horizontal edges. In particular, $Z_n$ is connected. Hence, it follows from Theorem~\ref{thm:Gonet} that $Z_n/Z_0$ is Galois. 
\end{proof}

\subsection{Proof of Theorem~\ref{thmB}}

Recall that $\mathcal{Y}_n$ denotes a connected component of $\mathcal{G}_l^p(n,N)$, which is defined in Definition~\ref{def:oriented-graph}.
\begin{theorem}\label{thm:structure1}
Let $\mathcal{Y}_n$ be a connected component of $\mathcal{G}_l^p(n,N)$.
    Assume that $\ell$ splits in the CM field of $\mathcal{Y}_n$. Then, $\mathcal{Y}_n$ is an undirected tectonic $l$-volcano.
\end{theorem}
\begin{proof}
Let $Y/X$ be a Galois covering of undirected connected graphs. If $Y$ is a tectonic $l$-volcano, the same is true for $X$. Thus, it suffices to prove the claim in the case $p^n>2$.
    Let $\mathfrak{Y}_n$ be the pre-image of $\mathcal{Y}_n$ in $\mathfrak{G}_l^p(n,N)$. It suffices to show that $\mathfrak{Y}_n$ is a tectonic $l$-volcano. Let $Z_n$ be the subgraph of depth zero vertices of $\mathcal{Y}_n$ and let $Z_0$ be the crater of $\mathcal{Y}_0$. By Lemma \ref{lem:galois}, $Z_n/Z_0$ is Galois.
    
    Let $\cO$ be the primitive order for some $(E,\iota)$ in $\VV(Z_0)$. Recall from the proof Lemma~\ref{lem:l-volcano} that $l$ is split in $\cO$. Therefore, each $l$-isogeny in $Z_n$ corresponds to one of the two prime ideals of the primitive order $\cO$ lying above $l$.  The colors of the edges in {$Z_n$} are chosen accordingly. In particular, the functions $\fb_Y$ and $\fg_Y$ given by Definition~\ref{def:color-function} commute.  Therefore, Lemma~\ref{lem:galois-tectonic} {applied to $\cY_n/\cY_0$} tells us that $Z_n$ is a tectonic crater. It now follows from Remark \ref{rem:level-preserving} that $\mathfrak{Y}_n$ is indeed a tectonic $l$-volcano. 
\end{proof}

\begin{theorem}\label{thm:structure2}
   Let $\mathcal{Y}_n$ be a connected component of $\mathcal{G}_l^p(n,N)$. Assume that $\ell$ is ramified in the CM field of $\mathcal{Y}_n$. Then, $\mathcal{Y}_n$ is an undirected $l$-volcano, and the crater is a cycle graph.
\end{theorem}
\begin{proof}
    The argument in the proof of Theorem~\ref{thm:structure1} carries over upon replacing Lemma~\ref{lem:galois-tectonic} by Lemma~\ref{lem:galois-non-tectonic}.
\end{proof}
We conclude this section with the following theorem on the inert case.

\begin{theorem}\label{thm:structure3}
  Let $\mathcal{Y}_n$ be a connected component of $\mathcal{G}_l^p(n,N)$.  Assume that $\ell$ is inert in the CM field of $\mathcal{Y}_n$. Then, $\mathcal{Y}_n$ is an undirected volcano, and the crater is disconnected.
\end{theorem}
\begin{proof}
     Remark \ref{rem:level-preserving} tells us that $\mathcal{Y}_n$ does not contain any horizontal edges {(see also \cite[Proposition 3.4]{arpin-et-all}}). Thus, the subgraph of depth-zero vertices is totally disconnected. It now follows {from the arguments in the proof of Theorem \ref{thm:structure1}} that $\mathcal{Y}_n$ is an $l$-volcano whose crater is totally disconnected. 
\end{proof}

\appendix

\section{An inverse problem for isogeny graphs of oriented supersingular elliptic curves}\label{app}
In \cite{pazuki}, Bambury--Campagna--Pazuki proved that if $G$ is a finite volcano graph, then there exist infinitely many distinct primes $p$ such that $G$ is isomorphic to a connected component of $\bG_l^p(0,1)$ that arises from ordinary elliptic curves defined over $\Fp$. We studied a similar question for abstract tectonic craters in \cite{LM1}. In particular, Theorem B of \textit{op. cit.} says that if $G$ is an abstract tectonic crater, then there exist infinitely many distinct prime numbers $p$ and $l$, and non-negative integers $N$ such that $G$ is isomorphic to a connected component of the crater of  $\bG_l^p(0,N)$ that arises from ordinary elliptic curves.

In this appendix, we extend the aforementioned theorem to craters of $\mathcal{G}_l^p(0,N)$.

\begin{theorem}\label{thm:inverse1}
    Let $G$ be an abstract tectonic crater of parameters $(\r,\s,\t,\cc)$. There exist infinitely many distinct prime numbers $l$ and $p$, and non-negative integers $N$ such that the crater of a connected component of $\fG_l^p(0,N)$  is isomorphic to $G$.
\end{theorem}
\begin{proof}
    Let $K$ be an imaginary quadratic field and let $\bE$ be an elliptic curve defined {over a finite abelian extension $K'$} with complex multiplication by the ring of integers $\mathcal{O}_K$ of $K$ {and such that $K'(\bE_{\textup{tors}})/K$ is an abelian extension}. {Such an extension exists following the discussion in \cite[\S II.1.4]{shalit}} . Let $L=K'(\bE[5])$ and let $p$ be a rational prime that is inert in $K/\Q$ and unramified in $L/\Q$. We fix three {distinct} prime numbers $N',M',Q'$ coprime to $5p\cdot\textup{cond}(\mathbf{E})$ (where $\textup{cond}(\mathbf{E})$ is the conductor of $\mathbf{E}$) that are split in $K$ and satisfy
    \begin{itemize}
        \item $N'\equiv 1 \pmod {\s}$,
        \item $M'\equiv 1\pmod{\t}$,
        \item $Q'\equiv 1\pmod{\r}$. 
    \end{itemize}
Note that {these three conditions do not contradict the fact that $N'$, $M'$ and $Q'$ should be split in $K$, which follows from the reciprocity law. The existence of these three primes follows from} Dirichlet's theorem on the infinitude of primes in arithmetic progression.

Let $\mathfrak{N}$ (resp. $\mathfrak{M}$, $\mathfrak{Q}$) be prime ideals of $\mathcal{O}_K$ lying above $N'$ (resp. $M'$, $Q'$). Let $L_1=L(E[N'M'\mathfrak{Q}])$ and $L_2=L(E[N'M'\overline{\mathfrak{Q}}])$. {By definition, the ray class field $K((5))$ lies in $L$. Furthermore, $\mathcal{O}^{{\times}}_K$ embeds into $(\mathcal{O}_{{K}}/(5))^\times$. Thus, $\Gal(L_1/L)$ is isomorphic to the subgroup of $ (\mathcal{O}_K/(5N'M'\mathfrak{Q}))^\times$ generated by elements that are congruent to $1$ modulo $(5)$.} Thus, \cite[Corollary II.1.7]{shalit} implies 
\begin{align}
&\Gal(L_1/L)\cong\Gal(K(5N'M'\mathfrak{Q})/K(5))\notag \\
&\cong (\mathcal{O}_K/\mathfrak{N})^\times \times (\mathcal{O}_K/\overline{\mathfrak{N}})^\times \times  (\mathcal{O}_K/\mathfrak{M})^\times \times (\mathcal{O}_K/\overline{\mathfrak{M}})^\times \times (\mathcal{O}_K/\mathfrak{Q})^\times,\label{eq:A1}\\
&\Gal(L_2/L)\cong\Gal(K(5N'M'\overline{\mathfrak{Q}})/K(5))    \notag\\
     &\cong (\mathcal{O}_K/\mathfrak{N})^\times \times (\mathcal{O}_K/\overline{\mathfrak{N}})^\times \times  (\mathcal{O}_K/\mathfrak{M})^\times \times (\mathcal{O}_K/\overline{\mathfrak{M}})^\times \times (\mathcal{O}_K/\overline{\mathfrak{Q}})^\times.\label{eq:A2}
\end{align}
By Tchebotarev's density theorem, there exist infinitely many prime ideals $\mathfrak{L}$ of $\mathcal{O}_K$ with the following properties:
\begin{itemize}
    \item $\mathfrak{L}$ splits in $L/K$.
    \item $\mathfrak{L}\neq \overline{\mathfrak{L}}$
    \item The image of the Frobenius of $\mathfrak{L}$ in $\Gal(L_1/L)$ {under the isomorphism \eqref{eq:A1}} is of the form $(a,1,1,b,d)$ with $\ord(a)=\s$, $\ord(b)=\t$ and $\ord(d^\s)=\r$.
    \item The {image of the} Frobenius of $\mathfrak{L}$ in $\Gal(L_2/L)$ {under the isomorphism \eqref{eq:A2}} is of the form $(a,1,1,b,d')$ with $\ord({d'}^\t)=\r$ and ${d'}^{\cc\t}=d^\s$.
\end{itemize}

Let $P$ be a primitive $5\mathfrak{N}\mathfrak{M}\mathfrak{Q}$-torsion point on $\bE$. Let $\sigma$ be the Frobenius of $\mathfrak{L}$ in $\Gal(L_1/K)$ and let $\tau$ be the Frobenius of $\overline{\mathfrak{L}}$. Let $H_1$ (resp. $H_2$) be the subgroup generated by $\sigma$ (resp. $\tau$). It follows from \cite[proof of Theorem 5.2]{LM1} that the orbit of $P$ under the action of $H_1$ (resp. $H_2$) contains $\r\s$ (resp. $\r\t$) elements. Furthermore, the intersection of the two orbits contains $\r$ elements and $\sigma^\s(P)=\tau^{\cc\t}(P)$. 

Note that $\mathbf{E}'=\mathbf{E}/\mathbf{E}[\mathfrak{L}]$ has complex multiplication by $\mathcal{O}_K$ as well {due to the fact that every endomorphism of $E$ sends an $\mathfrak{L}$-torsion point to another $\mathfrak{L}$-torsion point}. Indeed, every endomorphism of $\bE$ fixes the subgroup $\bE[\mathfrak{L}]$ and therefore defines an endomorphism of $\bE'$. Let $\mathfrak{p}$ be a prime ideal above $p$ in $L$. Then there is a horizontal isogeny \[\mathbf{E} \pmod{\mathfrak {p}}\to \mathbf{E}'\pmod{\mathfrak{p}}.\]
Thus, the ideals $\mathfrak{L}$ and $\overline{\mathfrak{L}}$ induce horizontal isogenies of $E:=\mathbf{E}\pmod{\mathfrak{p}}$.
Let $Q=P\pmod{\mathfrak{p}}$ and let $\iota$ be the natural map \[\textup{End}(\mathbf{E})\otimes \QQ\to \textup{End}(E)\otimes \QQ.\] Let $l$ be the rational prime below $\mathfrak{L}$. Let $\fY$ be the connected component of $\fG_l^p(0,N'M'Q')$ that contains $(E,\iota,Q)$. Then the crater of $\fY$ is an abstract tectonic {crater} with parameters $(\r,\s,\t,\cc)$.
\end{proof}

\begin{theorem}\label{thm:A2}
    Let $s$ be a non-negative integer. Then there exist infinitely many distinct prime numbers $p$ and $l$, and infinitely many non-negative integers $N$ such that the crater of a connected component of $\fG_l^p(0,N)$  is isomorphic to a directed cycle graph of length $s$. 
\end{theorem}
\begin{proof}
    Let $p\equiv 3\pmod{4}$ be a prime number. Let $l\ge 5$ be a prime number distinct from $p$ such that $-l$ is not a square modulo $p$. Assume, in addition, that $l-1$ has a prime factor different from $2$ and $3$. {The existence of such a prime follows from quadratic reciprocity, Dirichlet's theorem on primes in arithmetic progression and the fact that the set of natural numbers whose only prime divisors are $2$ and $3$ has density zero.}  Put $K=\QQ(\sqrt{-l})$. Let $\mathfrak{f}$ be a prime ideal of $\mathcal{O}_K$ coprime to $6$ such that $\mathfrak{f}|(\sqrt{-l}-1)$. Let $\mathbf{E}$ be an elliptic curve defined over  $L=K(\mathfrak{f})$ ({the ray class field of $K$ of conductor $\mathfrak{f}$}) that has good reduction outside $\mathfrak{f}$ and complex multiplication by ${\mathcal{O}_K}$ (see for example \cite[Chapter II]{shalit}).
    
    By class field theory, the ideal $(\sqrt{-l})$ splits completely in $L/K$. Pick a prime ideal $\mathfrak{N}$ of $\mathcal{O}_K$ such that $\ord(\sqrt{-l})$ in $(\mathcal{O}_K/\mathfrak{N})^\times$ is $s$. Let $P$ be a primitive $\mathfrak{Nf}$-torsion point on $\bE$. Let $H$ be the subgroup generated by the Frobenius of $(\sqrt{-l})$ in $\Gal(L(E[\mathfrak{N}])/K)$. Then the orbit of $P$  under the action $H$ is given by $\left\{[\sqrt{-l}]^nP:n\in\ZZ\right\}$, so it consists of $s$ elements. The remainder of the proof is analogous to that of Theorem~\ref{thm:inverse1}.
\end{proof}

\begin{remark}The reason why we only consider the crater is that the edges between vertices of different depths do not necessarily satisfy a simple pattern, as is illustrated by the examples given in \S\ref{S:def-volcano}. Note that we have also removed the level structure arising from the Verschiebung map from consideration. This is because our method relies on working with elliptic curves defined over fields of characteristic zero, for which the Verschiebung map does not apply.

Furthermore, note that the authors in \cite{pazuki} were able to solve an inverse problem for an arbitrary volcano graph, and not just for craters. After fixing a volcano graph, the prime $\ell$ is uniquely determined (by the degree of the vertices). One might ask whether it would be possible to improve Theorems~\ref{thm:inverse1} and \ref{thm:A2} to prove that for any given $l$ (resp. $p$), there are infinitely many $p$ (resp. $l$) such that the conclusions of the theorems hold. To do so, we would need to find an imaginary quadratic field $K$ and integers $N',M',Q$ that satisfy all the conditions used in the proof of Theorem~\ref{thm:inverse1}. However, our method does not work in such generality.
\end{remark}

\bibliographystyle{amsalpha}
\bibliography{references}

\newcommand{\etalchar}[1]{$^{#1}$}
\providecommand{\bysame}{\leavevmode\hbox to3em{\hrulefill}\thinspace}
\providecommand{\MR}{\relax\ifhmode\unskip\space\fi MR }
\providecommand{\MRhref}[2]{%
  \href{http://www.ams.org/mathscinet-getitem?mr=#1}{#2}
}
\providecommand{\href}[2]{#2}
\begin{thebibliography}{DLRV24}

\bibitem[ACL{\etalchar{+}}23]{arpin-et-all}
Sarah Arpin, Mingjie Chen, Kristin~E. Lauter, Renate Scheidler, Katherine~E. Stange, and Ha~T.~N. Tran, \emph{Orienteering with one endomorphism}, Matematica \textbf{2} (2023), no.~3, 523--582.

\bibitem[ACL{\etalchar{+}}24]{arpin-win}
\bysame, \emph{Orientations and cycles in supersingular isogeny graphs}, Research Directions in Number Theory: Women in Numbers V (2024), 25--86.

\bibitem[Arp24]{arpin}
Sarah Arpin, \emph{Adding level structure to supersingular elliptic curve isogeny graphs}, J. Théor. Nombres Bordeaux \textbf{36} (2024), no.~2, 405--443.

\bibitem[BCP22]{pazuki}
Henry Bambury, Francesco Campagna, and Fabien Pazuki, \emph{{Ordinary isogeny graphs over $\mathbb{F}_p$: The inverse volcano problem}}, 2022, to appear in Annali Scuola Norm Sup Pisa, available at arxiv: 2210.01086.

\bibitem[CK20]{colokohel}
Leonardo Col\`o and David Kohel, \emph{Orienting supersingular isogeny graphs}, J. Math. Cryptol. \textbf{14} (2020), no.~1, 414--437.

\bibitem[CL23]{codogni-lido}
Giulio Codogni and Guido Lido, \emph{Spectral theory of isogeny graphs}, 2023, preprint, arXiv:2308.13913.

\bibitem[CLG09]{CLG}
Denis~X. Charles, Kristin~E. Lauter, and Eyal~Z. Goren, \emph{Cryptographic hash functions from expander graphs}, J. Cryptology \textbf{22} (2009), no.~1, 93--113.

\bibitem[DFJP14]{DJP}
Luca De~Feo, David Jao, and J\'{e}r\^{o}me Pl\^{u}t, \emph{Towards quantum-resistant cryptosystems from supersingular elliptic curve isogenies}, J. Math. Cryptol. \textbf{8} (2014), no.~3, 209--247.

\bibitem[DLRV24]{DLRV}
Cédric Dion, Antonio Lei, Anwesh Ray, and Daniel Vallières, \emph{On the distribution of {I}wasawa invariants associated to multigraphs}, Nagoya Math. J. \textbf{253} (2024), 48--90.

\bibitem[dS87]{shalit}
Ehud de~Shalit, \emph{Iwasawa theory of elliptic curves with complex multiplication}, Perspectives in Mathematics, vol.~3, Academic Press, Inc., Boston, MA, 1987, $p$-adic $L$ functions.

\bibitem[EHL{\etalchar{+}}18]{EKHLMP}
Kirsten Eisentr\"{a}ger, Sean Hallgren, Kristin Lauter, Travis Morrison, and Christophe Petit, \emph{Supersingular isogeny graphs and endomorphism rings: reductions and solutions}, Advances in cryptology---{EUROCRYPT} 2018. {P}art {III}, Lecture Notes in Comput. Sci., vol. 10822, Springer, Cham, 2018, pp.~329--368.

\bibitem[GK21]{gorenkassaei}
Eyal~Z. Goren and Payman~L. Kassaei, \emph{{$p$}-adic dynamics of {H}ecke operators on modular curves}, J. Th\'{e}or. Nombres Bordeaux \textbf{33} (2021), no.~2, 387--431.

\bibitem[Gon21]{gonet-thesis}
Sophia~R. Gonet, \emph{Jacobians of {F}inite and {I}nfinite {V}oltage {C}overs of {G}raphs}, ProQuest LLC, Ann Arbor, MI, 2021, Thesis (Ph.D.)--The University of Vermont and State Agricultural College.

\bibitem[Gon22]{gonet22}
\bysame, \emph{Iwasawa theory of {J}acobians of graphs}, Algebr. Comb. \textbf{5} (2022), no.~5, 827--848.

\bibitem[GT87]{grosstucker}
Jonathan~L. Gross and Thomas~W. Tucker, \emph{Topological graph theory}, Wiley-Interscience Series in Discrete Mathematics and Optimization, John Wiley \& Sons, Inc., New York, 1987, A Wiley-Interscience Publication.

\bibitem[Iwa73]{iwasawa73}
Kenkichi Iwasawa, \emph{On {${\bf Z}_{l}$}-extensions of algebraic number fields}, Ann. of Math. (2) \textbf{98} (1973), 246--326.

\bibitem[Kat24]{kataoka}
Takenori Kataoka, \emph{Fitting ideals of {J}acobian groups of graphs}, Algebr. Comb. \textbf{7} (2024), no.~3, 597--625.

\bibitem[KM85]{katz-mazur}
Nicholas~M. Katz and Barry Mazur, \emph{Arithmetic moduli of elliptic curves}, Annals of Mathematics Studies, vol. 108, Princeton University Press, Princeton, NJ, 1985.

\bibitem[KM23]{KM}
Sören Kleine and Katharina Müller, \emph{On the non-commutative {I}wasawa main conjecture for voltage covers of graphs}, 2023, to appear in Israel J. Math., available at arXiv:2307.15395.

\bibitem[KM24]{KM1}
\bysame, \emph{On the growth of $\mathbf{Z}_p^l$-voltage covers of graphs}, Algbr. Comb. \textbf{7} (2024), 1011--1038.

\bibitem[Koh96]{kohel}
David~Russell Kohel, \emph{Endomorphism rings of elliptic curves over finite fields}, ProQuest LLC, Ann Arbor, MI, 1996, Thesis (Ph.D.)--University of California, Berkeley.

\bibitem[LM23]{LM2}
Antonio Lei and Katharina Müller, \emph{On towers of isogeny graphs with full level structure}, 2023, to appear in Res. Math. Sci., available at arXiv:2309.00524.

\bibitem[LM24]{LM1}
\bysame, \emph{On ordinary isogeny graphs with level structure}, Expo. Math. \textbf{42} (2024), no.~5, 125589.

\bibitem[LV23]{leivallieres}
Antonio Lei and Daniel Valli\`eres, \emph{The non-{$\ell$}-part of the number of spanning trees in abelian {$\ell$}-towers of multigraphs}, Res. Number Theory \textbf{9} (2023), no.~1, Paper No. 18, 16.

\bibitem[MV23]{vallieres2}
Kevin McGown and Daniel Valli\`eres, \emph{On abelian {$\ell$}-towers of multigraphs {II}}, Ann. Math. Qu\'{e}. \textbf{47} (2023), no.~2, 461--473.

\bibitem[MV24]{vallieres3}
\bysame, \emph{On abelian {$\ell$}-towers of multigraphs {III}}, Ann. Math. Qu\'{e}. \textbf{48} (2024), no.~1, 1--19.

\bibitem[Onu21]{onuki}
Hiroshi Onuki, \emph{On oriented supersingular elliptic curves}, Finite Fields Appl. \textbf{69} (2021), Paper No. 101777, 18.

\bibitem[PV23]{PV}
Riccardo Pengo and Daniel Vallières, \emph{{Spanning trees in $\mathbb{Z}$-covers of a finite graph and Mahler measures}}, 2023, to appear in J. Aust. Math. Soc., available at arXiv:2310.15619.

\bibitem[Rod19]{thesis-roda}
Megan Roda, \emph{{Supersingular Isogeny graphs with level $N$ structure and path problems on ordinary isogeny graphs}}, 2019, Master thesis, McGill University, https://escholarship.mcgill.ca/concern/theses/c247dx821.

\bibitem[RV22]{anwesh-daniel}
Anwesh Ray and Daniel Valli\'eres, \emph{{An analog of Kida's formula in graph theory}}, 2022, preprint, arXiv 2209.04890.

\bibitem[Ser77]{serre}
Jean-Pierre Serre, \emph{Arbres, amalgames, {${\rm SL}\sb{2}$}}, Ast\'erisque, vol. No. 46, Soci\'et\'e{} Math\'ematique de France, Paris, 1977, Avec un sommaire anglais, R\'edig\'e{} avec la collaboration de Hyman Bass.

\bibitem[Val21]{vallieres}
Daniel Valli\`eres, \emph{On abelian {$\ell$}-towers of multigraphs}, Ann. Math. Qu\'{e}. \textbf{45} (2021), no.~2, 433--452.

\bibitem[XZQ24]{XZQ}
Guanju Xiao, Zijian Zhou, and Longjiang Qu, \emph{Oriented supersingular elliptic curves and {E}ichler orders}, Finite Fields Appl. \textbf{100} (2024), 102501.

\end{thebibliography}

\end{document}